\documentclass[11pt,reqno,draft]{amsart} 
\usepackage[a4paper,left=24mm,right=24mm,top=32mm,bottom=30mm]{geometry} 
\usepackage{amsfonts}
\usepackage{amssymb}
\usepackage{amsfonts}
\usepackage{amssymb}
\usepackage{amsthm}
\usepackage{amsmath}
\usepackage{amsthm}
\usepackage[utf8]{inputenc}
\usepackage{comment}
\usepackage{color}
\usepackage[usenames,dvipsnames,svgnames,table]{xcolor}
\newif\iflong
\longtrue

\theoremstyle{plain}
\newtheorem{teo}{Theorem}[section]
\newtheorem{Lemma}[teo]{Lemma}
\newtheorem{Proposition}[teo]{Proposition}

\theoremstyle{definition}
\newtheorem{Definition}[teo]{Definition}

\theoremstyle{remark}
\newtheorem{Remark}[teo]{Remark}

\numberwithin{equation}{section}

\def\div{\operatornamewithlimits{div}\nolimits}
\def\eps{\varepsilon}

\def\Chi{\mathcal X}

\newcommand{\LL}{\mathcal{L}}

\def\loc{_{\operatorname{loc}}}

\def\R{\mathbb R}

\newcommand{\aaa}{a}
\newcommand{\aprime}{\alpha }

\renewcommand{\d}{\delta }
\newcommand{\D }{\Delta }

\newcommand{\e }{\varepsilon }
\newcommand{\Id}{Id }
\newcommand{\g }{\gamma}

\newcommand{\G }{\Gamma }

\renewcommand{\L }{\Lambda }

\newcommand{\rh }{\rho }

\renewcommand{\t }{\tau }

\newcommand{\intbar}{\etaathop{\int\etaakebox(-13.5,0){\rule[4pt]{.7em}{0.3pt}}%
\kern-6pt}\nolimits}

\newcommand{\be}{\begin{equation}}
\newcommand{\ee}{\end{equation}}
\newcommand{\bea}{\begin{equation*}}
\newcommand{\eea}{\end{equation*}}
\newcommand{\op}{\langle}
\newcommand{\cl}{\rangle}

\newcommand{\dr}{\dot{r}}
\newcommand{\drh}{\dot{\rho}}
\newcommand{\totaleo}{\frac{d}{d \epsilon}{\Big|}_{\epsilon=0}}
\newcommand{\partialeo}{\partial_{\epsilon}|_{\epsilon=0}}
\newcommand{\partialeoo}{\partial^2_{\epsilon}{\Big|}_{\epsilon=0}}

\newcommand{\grad}{\nabla}
\newcommand{\hgrad}{\hat{\nabla}}

\theoremstyle{plain}

\newtheorem{lemma}[teo]{Lemma}
\newtheorem{prop}[teo]{Proposition}
\newtheorem{cor}[teo]{Corollary}

\theoremstyle{definition}

\newtheorem{lem}{Lemma}[section]

\theoremstyle{remark}

\def\loc{_{\operatorname{loc}}}
\def\div{\operatornamewithlimits{div}\nolimits}
\def\tr{\operatornamewithlimits{tr}\nolimits}
\def\R{{{\mathbb R}}}
\def\SS{{{\mathbb S}}}
\def\uuu{{{\boldsymbol u}}}
\def\PPPhi{{{\boldsymbol \phi}}}
\def\mmmu{{{\boldsymbol \mu}}}
\def\SSSigma{{{\boldsymbol \Sigma}}}
\def\SSS{{{\mathcal S}}}
\def\nablatx{{{\nabla'}}}
\def\M{{{\mathcal M}}}
\def\Ha{{{\mathcal H}}}
\def\H{{{\mathcal H}}}

\def\eps{\varepsilon}

\def\AAA{{\mathrm A}}

\def\HHH{{\mathrm H}}
\def\hhh{{\mathrm h}}

\def\hmu{\hat{\mu}}

\def\bmu{\bar{\mu}}
\def\hmu{\hat{\mu}}

\def\rd{\textrm{d}}

\renewcommand{\d}{\delta }

\newcommand{\N}{\mathbb{N}}
\newcommand{\spt}{\textrm{spt}}

\def\be{\begin{equation}}
\def\ee{\end{equation}}
\def\bea{\begin{eqnarray*}}
\def\bean{\begin{eqnarray}}
\def\eean{\end{eqnarray}}
\def\eea{\end{eqnarray*}}

\begin{document}
\title
[Mean curvature flow action functional]
{Variational analysis of a mean curvature flow action functional}

\author[A. Magni]
{Annibale Magni}
\address[Annibale Magni]{Universit\"at Freiburg, Eckerstr. 1,
79104 Freiburg im Breisgau (Germany).}
\email[]{annibale.magni@math.uni-freiburg.de}
\author[M. R\"oger]
{Matthias R\"oger}
\address[Matthias R\"oger]{Technische Universit\"at Dortmund, Vogelpothsweg 87,
44227 Dortmund (Germany).}
\email[]{matthias.roeger@math.tu-dortmund.de}



\begin{abstract}
We consider the reduced Allen--Cahn action functional, which appears as the sharp interface limit of the Allen--Cahn action functional and can be understood as a formal action functional for a stochastically perturbed mean curvature flow. For suitable evolutions of (generalized) hypersurfaces this functional consists of the sum of the squares of the mean curvature and the velocity vectors, integrated over time and space. For given initial and final conditions we investigate the corresponding action minimization problem. We give a generalized formulation and prove compactness and lower-semicontinuity properties of the action functional. Furthermore we characterize the Euler--Lagrange equation for smooth stationary points and investigate  conserved quantities. Finally we present an  explicit example and consider concentric spheres as initial and final data and characterize in dependence of the given time span the properties of the minimal rotationally symmetric connection.
 \end{abstract}

\date{\today}

\maketitle

\section{Introduction}
Action functionals arise in large deviation theory as the lowest order in a small noise expansion for stochastically perturbed ODEs and PDEs. For a given deterministic path the corresponding value of the action-functional  is related to the probability that solutions of the stochastic dynamics are close to that path. For prescribed initial and final states an action minimizer may be associated with a most likely connecting path.

As a formal approximation of a stochastic mean curvature flow evolution we consider the Allen--Cahn equation perturbed by additive noise, i.e.
\begin{gather}
  \eps \partial_t u \,=\, \eps\Delta u -\frac{1}{\eps}W'(u) +
  \sqrt{2\gamma} \eta.
  \label{eq:stoch-AC}
\end{gather}
Here $\eps,\gamma>0$ are the interface thickness and noise-intensity parameter, $W$ is a fixed double-well potential, and $\eta$ describes a time-space white noise. As this equation admits in general only in one space dimension function-valued solutions a regularization for the noise is necessary.

In \cite{FaJo82} for one space dimension, and in \cite{Feng06},\cite{KORV} for higher dimensions the Allen--Cahn action functional was identified as the functional
\begin{gather}
 \tilde{\SSS}_\eps(u)\,:=\, \int_0^T\int_\Omega \Big(\sqrt{\eps}\partial_t u
  +\frac{1}{\sqrt{\eps}}\big(-\eps\Delta u 
  +\frac{1}{\eps}W^\prime(u)\big)\Big)^2\,dx\,dt. \label{def:action-AC-1}
\end{gather}
Computing the square and observing that the mixed term is a time derivative one obtains that for fixed initial and final data the action minimization problem is equivalent to the minimization of the functional
\begin{gather}
 {\SSS}_\eps(u)\,:=\, \int_0^T\int_\Omega {\eps}(\partial_t u)^2
  +\frac{1}{\eps}\big(-\eps\Delta u 
  +\frac{1}{\eps}W^\prime(u)\big)^2\,dx\,dt. \label{def:action-AC}
\end{gather}
In a series of papers \cite{ERVa04,FHSv04,KORV,KRT,RTon07,MuRoe09} reduced action functionals, defined as the sharp interface limit $\eps\to 0$ of $\tilde{\SSS}_\eps$ or $\SSS_\eps$, have been considered. In \cite{KORV} it was shown that
families  $(\Sigma_t)_{t\in (0,T)}$ of smoothly -- up to
finitely many `singular times' -- evolving smooth hypersurfaces can be
approximated with finite action $\tilde{\SSS}_\eps$. At the singular times a new
component is created in form of a double interface, which along the
subsequent evolution generates a `new phase'. For such evolutions a reduced action was derived that reads
\begin{align}
  \tilde{\SSS}_0(\Sigma)\,&:=\, c_0\int_0^T\int_{\Sigma_t} \big|v(x,t) -
  \HHH(x,t)\big|^2 \,d\Ha^{n}(x)dt\, + 4\tilde{\SSS}_{0,nuc}(u),
  \label{def:red-action-tlde} \\
  \tilde{\SSS}_{0,nuc}(u)\,&:=\, 2c_0 \sum_{i}\Ha^{n}(\Sigma_i), \label{def:AK-nuc-tlde}
\end{align}
where $n+1$ is the space dimension, $\Sigma_i$ denotes the $i^{th}$ component of $\Sigma$ at any time a new interface is nucleated, $v$ denotes the normal velocity vector for
the evolution $(\Sigma_t)_{t\in (0,T)}$, $H(t,\cdot)$ denotes the
mean curvature vector of $\Sigma_t$ and the constant $c_0$ depends only on the choice of the function $W$. The corresponding reduced action functional for the functionals $\SSS_\eps$ is given by
\begin{align}
  {\SSS}_0(\Sigma)\,&:=\, c_0\int_0^T\int_{\Sigma_t} \Big(|v(x,t)|^2 +
  |\HHH(x,t)|^2\Big) \,d\Ha^{n}(x)dt\, + 2{\SSS}_{0,nuc}(u),
  \label{def:red-action} \\
  {\SSS}_{0,nuc}(u)\,&:=\, 2c_0 \sum_{i}\Ha^{n}(\Sigma_i), \label{def:AK-nuc}
\end{align}
where the summation in the last line is now over the singular times at which nucleation or annihilation occur and where $\Sigma_i$ denotes the nucleated and annihilated components. In \cite{RTon07}, in the case of one space dimension, a generalization of ${\SSS}_0$ has been introduced and the Gamma convergence of $\SSS_\eps$ has been proved.
In \cite{MuRoe09} a general compactness statement for the sharp interface limit of sequences with bounded action $\tilde{\SSS}_\eps$ and initial or final data with uniformly controlled diffuse surface area has been shown. Moreover, a generalized reduced action functional has been proposed and a lower bound estimate has been proved. 

It is well-known \cite{AlCa79,RuSK89,MoSc90,EvSS92,Il:93g} that solutions of the Allen--Cahn equation converge to the evolution by mean curvature flow of phase boundaries. Therefore, the reduced action functional can formally be considered as a mean curvature flow action functional, although at present no rigorous connection to a suitable stochastically perturbed mean curvature flow is known. The goal of this paper is to study such formal mean curvature flow action functional for evolutions of generalized hypersurfaces. We restrict here to a suitable generalization of the functional $\SSS_0$ defined above, that has the nice property of being invariant under time-inversion. Independent of the question whether this functional in fact represents an action functional, the variational analysis helps to gain a better understanding of the behavior of the Allen--Cahn action functional itself. The variational problem for evolutions of surfaces has some interest in its own as it extends classical shape optimization problems for surfaces to the dynamic case.
The regular part of the functional $\SSS_0$ consists of the sum of a Willmore energy part and a velocity part. The Willmore functional has been studied intensively over the last decades, see for example \cite{Will65,simon2,Simo01,kuschat2,KuSc04,MaNe12} and is still an active field of geometric analysis. The minimization of the velocity part for given initial and final states is connected to an $L^2$-geodesic distance between these states. It has been shown in \cite{MiMu06} that this distance degenerates and is always zero; minimizing sequences use highly curved structures. By the addition of the Willmore term the functional $\SSS_0$ penalizes such evolutions and therefore represents a specific regularization (that however takes not the form of of Riemannian distance). In the minimization of the action functional we therefore see an interesting interplay of a stationary and dynamic contribution.

In this paper we begin a variational study of the reduced Allen--Cahn action functional. Our first goal is a compactness and lower semicontinuity result that allows for the application of the direct method of the calculus of variations, implying in particular the existence of minimizers. It is however \textit{a priori} not clear in what class of evolutions such a result can be achieved. In the class of smooth evolutions, for which the nucleation part in $\SSS_0$ drops out, a uniform bound on the action for a (minimizing) sequence does not provide sufficient control to derive a compactness statement in this class. In Section \ref{sec:generalized} we therefore provide a new generalized formulation in a specific class of evolutions of surface area measures and show in Section \ref{sec:cpct} compactness and lower semicontinuity properties for uniformly action bounded sequences of generalized evolutions. Lower-semicontinuity properties have not been shown in previous formulations of reduced Allen--Cahn action functionals and represent one main contribution of the current paper. The analysis of generalized evolutions and the application of the direct method of variations in the first part of our paper is complemented by the study of properties of smooth stationary points for the action functional. We derive the Euler--Lagrange equation for the action-minimization problem (Section \ref{sec:EL}) and study in Section \ref{sec:symm} conserved quantities, which reveals some analogies with Lagrangian mechanics. 
In Section \ref{sec:spher} we finally consider as a specific example the problem of finding the action-optimal connection between two concentric circles. We characterize minimizer in the class of rotationally symmetric solutions and their minimality properties with respect to the full class of smooth evolutions. The behavior turns out to be very different depending on the time-span given to connect the initial and the final state. 
\iflong
In the two appendices we collect some notation and results in geometric measure theory and in differential geometry which will be used throughout the paper.
\fi
\subsection*{General notation}
Let $n\in\N$ be fixed and consider for $T>0$ the space-time domain $Q_T := \R^{n+1} \times (0,T)$. 
For a function $\eta\in C^1(Q_T)$ we denote by $\nabla \eta$, $\partial_t \eta$, $\nablatx \eta$ the gradient with respect to the spatial variables, the time derivative, and the space-time gradient, respectively. In particular we have $\nablatx \eta = (\grad \eta,\partial_t \eta)^T$.\\
For a function $u\in BV(Q_T)$, we denote by $\nabla u, \partial_t u, \nablatx u$ the signed measures associated with the distributional derivative of $u$ in the $x,t$, and $(x,t)$-variables, respectively. With $|\nabla u|, |\partial_t u|, |\nablatx u|$ we denote the corresponding total variation measures.
For a family of Radon measures $(\mu_t)_{t\in (0,T)}$ we denote by
$\mu=\mu_t\otimes \LL^1$ the product measure, i.e.
\begin{gather*}
	\mu(\eta)\,=\, \int_0^T \mu_t(\eta(\cdot,t))\,dt\quad\text{ for all }\eta\in C^0_c(Q_T).
\end{gather*}
Throughout the paper we identify an integral $n$-varifold $V$ with its associated weight-measure $\mu=\mu_V$. For notation on geometric measure theory we refer to \iflong
the Appendix and to
\fi
the book of Simon \cite{Simo83}.


\subsection*{Acknowledgment}
We thank Stephan Luckhaus for sharing his insight on weak velocity formulations for evolving measures
and Luca Mugnai for stimulating discussions on the subject. 

This work was supported by the DFG Forschergruppe 718 \emph{Analysis and Stochastics in Complex Physical Systems}.

\section{Generalized action functional}
\label{sec:generalized}
Since in general a smooth minimizing sequence for the functional $\SSS_0$ does not necessarily converge to a smooth evolution (even up to finitely many singular times) we need to define a suitable class of generalized evolutions and a suitably generalized formulation for the action functional in that class in order to have lower semicontinuity and compactness for uniformly generalized action bounded evolutions.
We first recall the definition of $L^2$-flows \cite{MuRoe09} and in particular a characterization of velocity for certain evolutions of varifolds.
\begin{Definition}\label{def:L2-flow}
Let $T>0$ be given. Consider a family $\mmmu=(\mu_t)_{t\in (0,T)}$ of Radon measures on $\R^{n+1}$ and associate to $\mmmu$ the product measure $\mu\,:=\, \mu_t\otimes \LL^1$.
We call $\mmmu$ an \emph{$L^2$-flow} if the following properties hold:\\[1ex]
For almost all $t\in (0,T)$
\begin{align}
	&\mu_t\text{ is an integral $n$-varifold with }\sup_{0<t<T}\mu_t(\R^{n+1})<\infty, \label{eq:ass-gen-evol1}\\
	&\mu_t \text{ has weak mean curvature }H\in L^2(\mu_t). \label{eq:ass-gen-evol2}
\intertext{The evolution $\mmmu$ has a generalized normal velocity $v\in L^2(\mu;\R^{n+1})$, i.e.}
	&t\mapsto \mu_t(\psi)\quad\text{ is of bounded variation in }(0,T) \text{ for all }\psi\in C^1_c(\R^n), \label{eq:ass-gen-evol6}\\
	&v(x,t) \perp T_x \mu_t \quad \text{ for }\mu\text{-almost all }(x,t) \in Q_T, \label{eq:ass-gen-evol7}\\
	&\sup_{\eta} \Big| \int_{Q_T} (\partial_t \eta + \grad \eta \cdot v ) \rd
		\mu_t \rd t \Big| \,<\,\infty, \label{eq:gen-velo}
\end{align}
where the supremum is taken over all $\eta\in C^1_c(Q_T)$ with $|\eta|\leq 1$.
\end{Definition}
The evolution of measures $t\mapsto \mu_t(\psi)$, $\psi\in C^1_c(\R^{n+1})$ will only be controlled in $BV((0,T))$ and limit points therefore may have jumps in time. Thus, in order to formulate initial and final conditions, we need to complement the evolution of measures by an evolution of phases. For the action minimization problem we will therefore consider the following class of generalized evolutions.
\begin{Definition}\label{def:gen-evol}
Let $T>0$ and two open bounded sets $\Omega(0)$ and $\Omega(T)$ in $\R^{n+1}$ with finite perimeter be given. Let $\M=\M(T,\Omega(0),\Omega(T))$ be the class of tuples $\SSSigma=(\mmmu, \uuu)$, $\mmmu=(\mu_t)_{t\in (0,T)}$, $\uuu=(u(\cdot,t))_{t\in [0,T]}$, with the following properties:\\[1ex]
The evolution $\mmmu$ is an $L^2$-flow in the sense of Definition \ref{def:L2-flow}.\\
For almost all $t\in (0,T)$
\begin{align}
	&u(\cdot,t)\in BV(\R^{n+1},\{0,1\}), \label{eq:ass-gen-evol3}\\
	&|\nabla u(\cdot,t)|\,\leq\, \mu_t, \label{eq:ass-gen-evol4}
\end{align}
and $\uuu$ attains the initial and final data
\begin{align}
	& u(\cdot,0) \,=\, \Chi_{\Omega(0)},\quad u(\cdot,T) \,=\, \Chi_{\Omega(T)}.
	\label{eq:ass-cpct1}
\end{align}
The evolution $\uuu$ of phases satisfies $u\in C^{\frac{1}{2}}([0,T]; L^1(\R^{n+1}))$ and
\begin{align}
	& \int_{Q_T} \partial_t\eta(x,t)u(x,t)\,dx\,dt \,=\, \int_{Q_T} \eta (x,t)v(x,t)\cdot \nu(x,t) \,d|\nabla u(\cdot,t)|\,dt \label{eq:ass-gen-evol12}
\end{align}
for all $\eta\in C^1_c(Q_T)$, where $v$ is the generalized velocity of $\mmmu$ and where $\nu(\cdot,t)$ denotes the generalized inner normal on $\partial^*\{u(\cdot,t)=1\}$.
\end{Definition}
The property \eqref{eq:ass-gen-evol12} yields the following estimates.
\begin{lem}\label{lem:td-ul}
For $\SSSigma\in  \M$ as above we have that $u\,\in\, C^{\frac{1}{2}}([0,T];L^p(\R^{n+1}))$ for all $1\leq p<\infty$. For almost any $0\leq t_1\leq t_2\leq T$
\begin{align}
	\int_{\R^{n+1}} |u(x,t_2) - u(x,t_1)|\,dx \,&\leq\,  \| v\|_{L^2(\mu)}(t_2-t_1)^\frac{1}{2}\Big(\sup_{t_1<t<t_2}\mu_t(\R^{n+1})\Big)^{\frac{1}{2}} \label{eq:td-u-2}
\end{align}
holds. Moreover, $u\in BV(Q_T)$ with
\begin{align}
	\left(|\nabla u| + |\partial_t u|\right)(Q_T) \,&\leq\, 2T\sup_{0<t<T}\mu_t(\R^{n+1}) +  \int_{Q_T} |v|^2\,d\mu.
	\label{eq:td-u-11}
\end{align}
\end{lem}
\begin{proof}
First we deduce from \eqref{eq:ass-gen-evol12} that for any $\varphi\in C^1_c((0,T))$ and any $\psi\in C^1_c(\R^{n+1})$
\begin{align*}
	\Big|\int_0^T \partial_t \varphi(t) \int_{\R^{n+1}} u(x,t)\psi(x)\,dx\,dt\Big| \,=\, \Big|\int_0^T \varphi(t)\int_{\R^{n+1}}  \psi (x)v(x,t)\cdot \nu(x,t) \,d|\nabla u(\cdot,t)|\,dt\Big|.
\end{align*}
Hence the function $t\mapsto \int_{\R^{n+1}} u(x,t)\psi(x)\,dx$ belongs to $W^{1,2}((0,T))$ and for almost all $0<t_1<t_2<T$ we have
\begin{align*}
	\Big|\int_{\R^{n+1}} \big(u(x,t_2) -u(x,t_1)\big)\psi(x)\,dx\Big| \,\leq\, \|v\|_{L^2(\mu)} (t_2-t_1)^{\frac{1}{2}}\left(\sup_{t_1 < t < t_2}\mu_t(\R^{n+1})\right)^{\frac{1}{2}}\|\psi\|_{C^0_c(\R^{n+1})}.
\end{align*}
Since $u(x,t_2) -u(x,t_1)\in BV(\R^{n+1},\{-1,0,1\})$, taking the supremum over $\psi\in C^0_c(\R^{n+1})$ with $\|\psi\|\leq 1$ yields \eqref{eq:td-u-2}. Since $|u(x,t_2) -u(x,t_1)|\leq 1$ almost everywhere we deduce that $u\,\in\, C^{\frac{1}{2}}([0,T];L^p(\R^{n+1}))$ for all $1\leq p<\infty$. From \eqref{eq:ass-gen-evol4}, \eqref{eq:ass-gen-evol12} one gets $u\in BV(Q_T)$ and \eqref{eq:td-u-11}.
\end{proof}

In the class $\M$ we next define a generalized action functional.
\begin{Definition}\label{def:gen-action}
For $\SSSigma\in \M$, $\SSSigma=(\mmmu,\uuu)$ as above we define
\begin{align}
	\SSS(\SSSigma)	\,&:=\,  \SSS_+(\SSSigma) +\SSS_-(\SSSigma),  \label{eq:def-gen-action}\\
	\SSS_+(\SSSigma)	\,&:=\, \sup_{\eta} \Big[ 2|\nabla u(\cdot,T)|(\eta(\cdot,T)) - 2 |\nabla u(\cdot,0)|(\eta(\cdot,0))  \notag\\
	&\qquad\qquad + \int_{Q_T} -2\big(\partial_t\eta + \nabla\eta\cdot v\big) +(1-2\eta)_+\frac{1}{2}|v-H|^2\,d\mu_t\,dt \Big], \label{eq:def-gen-action+}\\
	\SSS_-(\SSSigma)	\,&:=\,  \sup_{\eta} \Big[ -2|\nabla u(\cdot,T)|(\eta(\cdot,T)) + 2 |\nabla u(\cdot,0)|(\eta(\cdot,0))  \notag\\
	&\qquad\qquad + \int_{Q_T} 2\big(\partial_t\eta + \nabla\eta\cdot v\big) +(1-2\eta)_+\frac{1}{2}|v+H|^2\,d\mu_t\,dt \Big], \label{eq:def-gen-action-}%
\end{align}
where the supremum is taken over all $\eta\in C^1(\R^{n+1}\times [0,T])$ with $0\leq\eta\leq 1$.
\end{Definition}
We remark that $\SSS$ is invariant under the time inversion $t\mapsto T-t$. Since in $\SSS_\pm$ the terms $(1-2\eta)_+\frac{1}{2}|v\mp H|^2$ are nonnegative, we observe that a bound on the action implies the generalized velocity property \eqref{eq:gen-velo} and, more precisely, the estimate
\begin{align}
	\Big|\int_{Q_T} \big(\partial_t\eta + \nabla\eta\cdot v\big) \,d\mu\Big| \,\leq\, \frac{1}{2}\SSS(\SSSigma) \label{eq:gen-velo-S}
\end{align}
for all $\eta\in C^1_c(Q_T)$ with $|\eta|\leq 1$.
By choosing $\eta=0$ in \eqref{eq:def-gen-action+},\eqref{eq:def-gen-action-} we further have that
\begin{align}
	\int_{Q_T}(|v|^2 + |H|^2)\,d\mu_t\,dt \,\leq\, \SSS(\SSSigma). \label{eq:class-S}
\end{align}
The functional $\SSS$ takes into account also jumps in the evolution of the generalized surface measures $t\mapsto \mu_t$ and actually generalizes the notion of action functional for the smooth case.
\begin{Proposition}\label{prop:consist}
Let $\SSSigma=(\mmmu,\uuu)$ be given by an evolution $(\Omega(t))_{t\in [0,T]}$ of open sets $\Omega(t)\subset \R^{n+1}$ as
\begin{align*}
	u(\cdot,t) \,=\, \Chi_{\Omega(t)}\qquad\text{ and }\quad \mu_t \,:=\, \Ha^n\lfloor \partial\Omega(t).
\end{align*}
Assume that $(\partial \Omega(t))_{t\in [0,T]}$ represents, outside of a set of possibly singular times $0=t_0 < t_1<\dots<t_k < t_{k+1}=T$, a smooth evolution of smooth hypersurfaces. Then
\begin{align}
	\SSS(\SSSigma) \,&=\, \int_0^T \int_{ \partial \Omega(t)} (|v(\cdot,t)|^2+|H(\cdot,t)|^2) \,d\Ha^n\,dt + 2\sum_{j=0}^{k+1} \sup_{\psi} |\mu_{t_j+}(\psi)-\mu_{t_j-}(\psi)|, \label{eq:consist}
\end{align}
where the supremum is taken over all $\psi\in C^1(\R^n)$ with $|\psi|\leq 1$ and where we have set $\mu_t := \Ha^n\lfloor \partial\Omega(0)$ for $t< 0$ and $\mu_t := \Ha^n\lfloor \partial\Omega(T)$ for $t>T$.
\end{Proposition}
\begin{proof}
We first compute that $\mu$-almost everywhere it holds
\begin{align}
	&-2\big(\partial_t\eta + \nabla\eta\cdot v\big) +(1-2\eta)_+\frac{1}{2}|v-H|^2\\
	=\, &-2\big(\partial_t\eta + \nabla\eta\cdot v -\eta v\cdot H\big) +(1-2\eta)_+\frac{1}{2}|v-H|^2 -2\eta v\cdot H. \label{eq:consist-1}
\end{align}
For the second term we observe that for $0\leq \eta\leq \frac{1}{2}$
\begin{align}
	(1-2\eta)_+\frac{1}{2}|v-H|^2 -2\eta v\cdot H \,=\, \frac{1}{2}|v-H|^2 - \eta (|v|^2+|H|^2) \,\leq\, \frac{1}{2}|v-H|^2 \label{eq:consist-2}
\end{align}
and for $\frac{1}{2}\leq \eta\leq 1$
\begin{align}
	(1-2\eta)_+\frac{1}{2}|v-H|^2 -2\eta v\cdot H \,=\, -2\eta v\cdot H \,\leq\, 2 |v\cdot H| \Chi_{\{v\cdot H <0\}} \,\leq\, \frac{1}{2}|v-H|^2. \label{eq:consist-3}
\end{align}
Moreover we have for any $0\leq j\leq k$ that
\begin{align}
	\int_{t_j}^{t_j+1}\int_{\R^{n+1}} 2\big(\partial_t\eta + \nabla\eta\cdot v -\eta v\cdot H\big) \,d\mu_t\,dt
	\,&=\, 2\int_{t_j}^{t_{j+1}} \frac{d}{dt} \Big(\int_{\partial \Omega(t)} \eta(\cdot,t)\,d\Ha^n\Big)\,dt \notag\\
	&=\, 2\Big( \lim_{t\nearrow t_{j+1}} \mu_{t}(\eta(\cdot,t)) - \lim_{t\searrow t_{j}} \mu_{t}(\eta(\cdot,t))\Big)
\end{align}
and therefore
\begin{align}
	& 2|\nabla u(\cdot,T)|(\eta(\cdot,T)) -  2|\nabla u(\cdot,0)|(\eta(\cdot,0)) - \int_{Q_T} 2\big(\partial_t\eta 
		+ \nabla\eta\cdot v -\eta v\cdot H\big) \,d\mu_t\,dt \notag \\
	\,=\,&  
	 2\sum_{j=0}^{k+1} \Big(\mu_{t_j+}(\eta(\cdot,t_j)) - \mu_{t_j-}(\eta(\cdot,t_j)\Big) \notag \\
	\leq\,& 2\sum_{j=0}^{k+1} \sup_{\psi} \big(\mu_{t_j+}(\psi)-\mu_{t_j-}(\psi)\big), \label{eq:consist-4}
\end{align}
where  the supremum is taken over all $\psi\in C^1(\R^n)$ with $0\leq \psi \leq 1$.  Together with \eqref{eq:consist-2} and \eqref{eq:consist-3} we deduce
\begin{align}
	\SSS_+(\SSSigma) \,\leq\, \frac{1}{2} \int_0^T \int_{ \partial \Omega(t)} (|v(\cdot,t)|^2+|H(\cdot,t)|^2)\,d\Ha^n\,dt + 2\sum_{j=0}^{k+1} \sup_{\psi} \big(\mu_{t_j+}(\psi)-\mu_{t_j-}(\psi)\big)_+.\label{eq:consist-5}
\end{align}
On the other hand, by choosing $\eta=0$ except in an arbitrary small neighborhood of the $t_j$'s and by choosing $\eta(\cdot,t_j)$ to approximate the supremum in $\sup_{\psi} \big(\mu_{t_j+}(\psi)-\mu_{t_j-}(\psi)\big)$ we see that we have in fact equality in \eqref{eq:consist-5}. Similarly we derive
\begin{align}
	\SSS_-(\SSSigma) \,=\, \frac{1}{2} \int_0^T \int_{ \partial \Omega(t)} (|v(\cdot,t)|^2+|H(\cdot,t)|^2)\,d\Ha^n\,dt + 2\sum_{j=0}^{k+1} \sup_{\psi} \big(\mu_{t_j-}(\psi)-\mu_{t_j+}(\psi)\big)_+ \label{eq:consist-6}
\end{align}
where  the supremum is taken over all $\psi\in C^1(\R^n)$ with $0\leq \psi \leq 1$. Summing up this equality with \eqref{eq:consist-5} we finally obtain \eqref{eq:consist}.
 \end{proof}
The expression on the right-hand side of \eqref{eq:consist} corresponds to the definition $\SSS_0$ of the action functional for the (semi-)smooth case. 

The proof of Proposition \ref{prop:consist} shows in particular that $\SSS_+$ measures all the upward jumps of the measure evolution $t\mapsto \mu_t$ and  that $\SSS_-$ measures all the downward jumps.
\section{Compactness and lower-semicontinuity for uniformly action-bounded sequences}
\label{sec:cpct}
In this section we consider sequences of generalized evolutions that are uniformly bounded in action and constrained to fixed initial and final data.
The main results of this section are the following compactness and lower-semicontinuity statements.
\begin{teo} \label{thm:main}
Let $T>0$ and two open bounded sets $\Omega(0)$ and $\Omega(T)$ in $\R^{n+1}$ with finite perimeter be given. Consider a family of evolutions $(\SSSigma_l)_{l\in\N}$ in $\M(T,\Omega(0),\Omega(T))$ with
\begin{align}
	\SSS(\SSSigma_l)\,&\leq\, \Lambda \quad\text{ for all }l\in\N,\label{eq:ass-cpct3}
\end{align}
where $\Lambda>0$ is a fixed constant.

Then there exists a subsequence $l \to\infty$ (not relabeled) and a limit evolution $\SSSigma=(\mmmu,\uuu)\,\in\, \M(T,\Omega(0),\Omega(T))$, $\mmmu=(\mu_t)_{t\in (0,T)}$, $\uuu=(u(\cdot,t))_{t\in [0,T]}$, such that
\begin{align}
	u^l\,&\to\, u \qquad \text{ in }L^1(Q_T) \cap C^0( [0,T];L^1(\R^{n+1})), \label{eq:cpct-BV-1}\\
	\mu^l_t\,&\to\, \mu_t\quad \text{ for almost all }t\in (0,T) \text{ as integral varifolds on } \R^{n+1}, \label{eq:conv-mu-t}\\
	\mu^l \,&\to\, \mu \quad\text{ as Radon measures on } Q_T.\label{eq:conv-mu}
\end{align}
Moreover
\begin{align}
	\SSS(\SSSigma) \,&\leq\, \liminf_{l\to\infty} \SSS(\SSSigma_l) \label{eq:lsc-SSS}
\end{align}
holds. In particular, the minimum of $\SSS$ in $\M(T,\Omega(0),\Omega(T))$ is attained.
\end{teo}
In the  remainder of the section we prove Theorem \ref{thm:main}. The line of the proof follows closely the arguments of \cite{MuRoe09} which are themselves based on \cite{KRT,KORV}. However, the situation here is different, as we do not pass to the limit with phase field approximations but with a sequence of sharp interface evolutions. Moreover, our formulation of generalized action functional is different from that in \cite{MuRoe09}. Therefore all proofs need to be adapted. For most statements we give the detailed arguments but refer to the corresponding statement in \cite{MuRoe09}.  

From \eqref{eq:gen-velo-S}, \eqref{eq:class-S}, and \eqref{eq:ass-cpct3} we first obtain the uniform bounds
\begin{align}
	\int_{Q_T}(|v_l|^2 + |H_l|^2)\,d\mu^l_t\,dt \,&\leq\, \Lambda, \label{eq:class-S-l}\\
	\sup_{\eta\in C^1_c(Q_T)}\Big|\int_{Q_T} \big(\partial_t\eta + \nabla\eta\cdot v\big) \,d\mu^l_t\,dt\Big| \,&\leq\, \frac{1}{2}\Lambda \|\eta\|_{C^0_c(Q_T)}. \label{eq:gen-velo-S-l}
\end{align}
To the integral varifolds $(\mu^l_t)_{t\in (0,T)}$ we associate the product measures $\mu^l\,:=\, \mu^l_t\otimes \LL^1$.
We start with showing that the assumptions above induce a uniform bound for the area measures and that time differences of the area measures are controlled by means of the initial data and $\L$.
\begin{prop}\label{prop:area}\cite[Lemma 5.1]{MuRoe09}
For all $l\in\N$ we have
\begin{align}
	\sup_{t\in (0,T)} \mu^l_t(\R^{n+1})\,&\leq\, C(\Omega(0),T,\Lambda), \label{eq:area-1}\\
	\mu^l(Q_T)\,&\leq\, C(\Omega(0),T,\Lambda). \label{eq:area-2}
\end{align}
Moreover for all $\psi \in  C^1_c(\R^{n+1})$ the function $t\,\mapsto\, \mu^l_t(\psi)$ is of bounded variation in $(0,T)$ with
\begin{align}
	\sup_{l\in\N} |\partial_t \mu^l_t(\psi)|((0,T))\,&\leq\,	C(\Omega(0),T,\Lambda)\|\psi\|_{C^1_c(\R^{n+1})}. \label{eq:d-ar-2}
\end{align}
\end{prop}
\begin{proof}
Choosing $\eta(x,t)=\varphi(t)$ for $\varphi\in C^1_c((0,T))$, from \eqref{eq:gen-velo-S-l} we first deduce that $M_l: (0,T)\to\R^+_0$, $M_l(t):=\mu^l_t(\R^{n+1})$ is of bounded variation with
\begin{gather}
	|M_l'|((0,T))\,\leq\, \frac{\Lambda}{2}. \label{eq:ml-BV}
\end{gather}
Choosing $\eta(x,t)=\varphi(t)$, with $\varphi\in C^1([0,T])$ not necessarily compactly supported in $(0,T)$, we obtain from the definition of $\SSS$ that
\begin{align}
	\Big|\lim_{t\searrow 0} M_l(t) - \Ha^n(\partial^*\Omega(0))\Big| \,&\leq\, \frac{\Lambda}{2}, \label{eq:set-inital}\\
	\Big|\lim_{t\nearrow T} M_l(t) - \Ha^n(\partial^*\Omega(T))\Big| \,&\leq\, \frac{\Lambda}{2}. \label{eq:set-final}
\end{align}
Actually, setting
\begin{gather*}
	\varphi_k(t) \,:=\, 
	\begin{cases}
		1-kt &\text{ for } 0\leq t\leq \frac{1}{k}\\
		0 &\text{ otherwise}
	\end{cases}\,,
\end{gather*}
we obtain
\begin{align*}
	\Lambda \,&\geq\, -2 \Ha^n(\partial^*\Omega(0)) + 2k\int_0^{\frac{1}{k}} M_l(t)\,dt,\\
	\Lambda \,&\geq\, 2 \Ha^n(\partial^*\Omega(0)) - 2k\int_0^{\frac{1}{k}} M_l(t)\,dt.
\end{align*}
Thus,  taking the limit $k\to\infty$, \eqref{eq:set-inital} holds.
Similarly, one obtains \eqref{eq:set-final}.  Together with \eqref{eq:ml-BV} we then deduce that 
\begin{align*}
	\mu^l_t(\R^{n+1}) \,&\leq\, \Ha^n(\partial^*\Omega(0)) + \Lambda,
\end{align*}
holds, which proves \eqref{eq:area-1}. The estimate \eqref{eq:area-2} follows.

Next we fix $\psi\in C^1_c(\R^{n+1})$ and obtain from \eqref{eq:gen-velo-S-l} with $\eta(x,t)=\varphi(t)\psi(x)$ that $t\,\mapsto\, \mu^l_t(\psi)$ is of bounded variation in $(0,T)$ and that
\begin{align}
	|D\mu^l_t(\psi)|((0,T))\,&\leq\, \frac{1}{2}\Lambda \|\psi\|_{C^0_c(\R^{n+1})}+ \sup_{|\varphi|\leq 1}\Big|\int_0^T \varphi(t)\int_{\R^{n+1}} \nabla\psi\cdot v(\cdot,t)\,d\mu^l_t\,dt\Big| \notag\\
	&\leq\, \frac{1}{2}\Lambda \|\psi\|_{C^0_c(\R^{n+1})}+ \|\nabla\psi\|_{L^2(\mu)}\|v\|_{L^2(\mu)}\notag\\
	&\leq  \Big(\frac{1}{2}\Lambda  + \big(T\sup_{0<t<T}\mu^l_t(\R^{n+1})\big)^{1/2}\Lambda^{1/2}\Big)\|\psi\|_{C^1_c(\R^{n+1})}, \label{eq:mu-l-t-BV2}
\end{align}
where we have used \eqref{eq:class-S-l}.
Together with \eqref{eq:area-1} the estimate \eqref{eq:d-ar-2} follows.
\end{proof}
The previous proposition, Lemma \ref{lem:td-ul}, and \eqref{eq:class-S-l} yield the uniform bounds
\begin{align}
	\int_{\R^{n+1}} |u_l(x,t_2) - u_l(x,t_1)|\,dx \,&\leq\,  C(\Lambda,T,\Omega(0))(t_2-t_1)^\frac{1}{2}, \label{eq:td-u-2-l}\\
	\left(|\nabla u_l| + |\partial_t u_l|\right)(Q_T) \,&\leq\, C(\Lambda,T,\Omega(0)).
	\label{eq:td-u-11-l}
\end{align}

Combining Proposition \ref{prop:area} and Lemma \ref{lem:td-ul}, we obtain a compactness statement for the characteristic functions of the enclosed sets.
\begin{prop}\label{prop:cpct-phases}\cite[Prop 4.1]{MuRoe09}.
There exist a subsequence $l\to\infty$ (not relabeled) and a function $u\in BV(Q_T;\{0,1\})$, $u\in C^{\frac{1}{2}}([0,T];L^1(\R^{n+1}))$ such that
\eqref{eq:ass-cpct1} and \eqref{eq:cpct-BV-1} hold.
\end{prop} 
\begin{proof}
By \eqref{eq:td-u-11-l}, the compactness Theorem for BV functions ensures the existence of a subsequence $l\to\infty$ and of a function $u\in BV(Q_T)$, with $u^l\,\to\, u$ in $L^1(Q_T)$. In particular, $u(x,t)\in\{0,1\}$ for almost every $(x,t)\in Q_T$.

From \eqref{eq:td-u-2-l}, we deduce that $(u^l)_{l\in\N}$ is uniformly bounded in $C^{\frac{1}{2}}( [0,T];L^1(\R^{n+1})$. Moreover, by \eqref{eq:ass-gen-evol4} for $\SSSigma_l$ and \eqref{eq:area-1}, the family $\{ u_l(t)\,:\, l\in\N\}$ is relatively compact in $L^1(\R^{n+1})$ for almost any $t\in (0,T)$. Applying the Arzela-Ascoli Theorem we deduce that,  possibly after passing to another subsequence, $u^l\to u$ in $C^0( [0,T];L^1(\R^{n+1}))$, with $u\in C^{\frac{1}{2}}([0,T];L^1(\R^{n+1}))$.

The condition \eqref{eq:ass-cpct1} for $\SSSigma_l$ implies by \eqref{eq:cpct-BV-1} that $u$ attains the initial and final data.
\end{proof}
We next show a compactness statement for the evolution of the surface area measures.
\begin{prop}\label{prop:cpct-bdry}\cite[Prop 4.2]{MuRoe09}
There exists a subsequence $l\to\infty$ (not relabeled) and a family of Radon measures $(\mu_t)_{t\in (0,T)}$ on $\R^{n+1}$ such that
 \eqref{eq:ass-gen-evol6}, \eqref{eq:ass-gen-evol4}, 
\begin{align}
 	\mu^l_t\,&\to\, \mu_t\quad \text{ for all }t\in (0,T) \text{ as Radon measures on } \R^{n+1}, \label{eq:conv-mu-t-RM}
\end{align}
and \eqref{eq:conv-mu} hold. Moreover
\begin{align}
	\sup_{t\in [0,T]} \mu_t(\R^{n+1})\,&\leq\, C(\Omega(0),T,\Lambda) \label{eq:area-1-lim}
\end{align}
is satisfied.
\end{prop}
\begin{proof}
We first choose a countable family $(\psi_k)_{k\in\N}$ in $C^1_c(\R^{n+1})$ which is dense in $C^0_0(\R^{n+1})$ with respect to the supremum norm. By \eqref{eq:area-1} and \eqref{eq:d-ar-2}, we have that  for fixed  $k\in\N$ the family of functions $(t\mapsto \mu_t^l(\psi_k))_{l\in\N}$ is uniformly bounded in $BV(0,T)$. By a diagonal-sequence argument we obtain a
subsequence $l\to\infty$ and functions $m_k\in BV(0,T)$, $k\in\N$, such that
for all $k\in\N$
\begin{align}
	\mu^l_t(\psi_k)\,&\to\, m_k(t)\quad&&\text{ for almost-all }t\in (0,T),
 \label{eq:conv-mk}
 \\
 D\mu^l_t(\psi_k)\,&\to\, m_k'\quad&&\text{ as Radon measures on
 }(0,T). \label{eq:conv-mk-2} 
\end{align}
Let $S$ denote the countable set of times $t\in (0,T)$ where, for some $k\in\N$, the
measure $m_k'$ has an atomic part. 
\iflong
We claim that
\eqref{eq:conv-mk} holds on $(0,T)\setminus S$. To see this
we choose a point $t\in (0,T)\setminus S$
and a sequence of points $(t_j)_{j\in\N}$ in $(0,T)\setminus S$, such that
$t_j\nearrow t$ and \eqref{eq:conv-mk} holds for all $t_j$ (the case $t_j\searrow t$ can be treated analogously). We thus obtain
\begin{align}
  \lim_{j\to\infty} m_k'([t_j,t])\,&=\,0 &&\text{ for all }k\in\N,
  \label{eq:atom-mk}\\ 
  \lim_{l\to\infty} \partial_t \mu^l(\psi_k)([t_j,t])\,&=
  m_k'([t_j,t])\quad&&\text{ for all }k,j\in\N, \label{eq:conv-d-mk}
\end{align}
since $t_j,t\in (0,T)\setminus S$.
Moreover
\begin{align*}
  |m_k(t)-\mu^l_t(\psi_k)|\,&\leq\, 
  |m_k(t)-m_k(t_j)|+|m_k(t_j)-\mu^l_{t_j}(\psi_k)|
  +|\mu^l_{t_j}(\psi_k)- \mu^l_t(\psi_k)|\\
  &\leq\, |m_k'([t_j,t])| + |m_k(t_j)-\mu^l_{t_j}(\psi_k)| +
  |\partial_t \mu^l_t(\psi_k)([t_j,t])|
\end{align*}
Taking first $l\to\infty$ and then $t_j\nearrow t$, we deduce by
  \eqref{eq:atom-mk} and \eqref{eq:conv-d-mk} that \eqref{eq:conv-mk} holds for all
$k\in\N$ and all $t\in (0,T)\setminus S$.

Taking now an arbitrary $t\in (0,T)$ such that \eqref{eq:conv-mk} holds, \eqref{eq:area-1} ensures the existence of a subsequence $l\to\infty$ such that
\begin{gather}
  \mu^l_t\,\to\, \mu_t \quad\text{ as Radon-measures on
  }\R^{n+1}. \label{eq:conv-mu-l-t} 
\end{gather}
We deduce that $\mu_t(\psi_k)=m_k(t)$ and, since $(\psi_k)_{k\in\N}$ is
dense in $C^0_0(\R^{n+1})$, we can identify any limit of
$(\mu^l_t)_{l\in\N}$ and obtain \eqref{eq:conv-mu-l-t} for the
whole sequence selected in \eqref{eq:conv-mk}--\eqref{eq:conv-mk-2}, and for all $t\in (0,T)$, for which
\eqref{eq:conv-mk} holds.  This proves \eqref{eq:conv-mu-t-RM}.\\
For any $\psi\in C^0_0(\R^{n+1})$ the map $t\mapsto \mu_t(\psi)$ has no
jumps in $(0,T)\setminus S$ and for all $\varphi\in C^1_c((0,T))$ with $|\varphi|\leq 1$, by \eqref{eq:d-ar-2}, we have
\begin{align*}
	\Big| \int_{\R^{n+1}} \partial_t\varphi(t)\mu_t(\psi)\,dt\Big| \,&=\, \Big|\lim_{l\to\infty}\int_{\R^{n+1}} \partial_t\varphi(t)\mu^l_t(\psi)\,dt\Big| \\
	&\leq\, \liminf_{l\to\infty} |\partial_t \mu^l_t(\psi)|\,\leq\, C(\Omega(0),T,\Lambda).
\end{align*}
This proves \eqref{eq:ass-gen-evol6}.

By the Dominated Convergence Theorem we further conclude that for any $\eta\in
C^0_c(Q_T)$ 
\begin{align*}
  \int_{Q_T} \eta\,d\mu\,&=\, \lim_{l\to\infty}
  \int_{Q_T} \eta\,d\mu^l \\
  &=\,  \lim_{l\to\infty}\int_{Q_T} \eta(x,t)\,d\mu^l_t(x)\,dt\,=\, \int_{Q_T}
  \eta(x,t)\,d\mu_t(x)\,dt,
\end{align*}
which implies \eqref{eq:conv-mu}.
\fi
By \eqref{eq:cpct-BV-1} we have $u_l(\cdot,t)\to u(\cdot,t)$ in $L^1(\R^{n+1})$ as $l\to\infty$. By
\eqref{eq:ass-gen-evol4} for $\SSSigma_l$, \eqref{eq:conv-mu-t} and the lower-semicontinuity of the
perimeter under $L^1$-convergence we conclude
that $|\nabla u(\cdot,t)|\,\leq\, \mu_t$ holds, which proves \eqref{eq:ass-gen-evol4}.
\end{proof}
We next show that the measures $\mu_t, t\in (0,T)$, are integral varifolds with weak mean curvature in $L^2(\mu_t)$.
\begin{prop}\label{prop:mu-t}\cite[Thm. 4.3]{MuRoe09}
For any $t\in (0,T)$ the limit measure $\mu_t$ as in \eqref{eq:conv-mu-t} is an integral varifold with weak mean curvature $H(\cdot,t)\in L^2(\mu_t)$ and for almost all $t\in (0,T)$ 
\begin{align}
	\mu^l_t \,\to\, \mu_t\, (l\to\infty)\quad\text{ as varifolds} \label{eq:conv-mu-t-var}
\end{align}
holds, which proves \eqref{eq:conv-mu-t}.
The sequence $(\mu^l,\HHH_l)_{l\in\N}$ converges to $(\mu,\HHH)$ as measure function pairs, i.e.
\begin{align}
	\int_0^T \int_{M^l_t} \eta(.,t) \HHH_l(.,t)\,d\Ha^n\,dt \,\to\,  \int_{Q_T} \eta(.,t) \HHH(.,t)\,d\mu_t\,dt \label{eq:mu-H-conv}
\end{align}
holds for all $\eta\in C^0_c(\R^{n+2}_{0,T};\R^{n+1})$. We moreover have the estimates
\begin{align}
	\int_{\R^{n+1}} |\HHH(\cdot,t)|^2\,d\mu_t \,&\leq\, \liminf_{l\to\infty}  \int_{M^l_t} |\HHH_l(\cdot,t)|^2\,d\Ha^n\qquad\text{ for almost all }t\in (0,T), \label{eq:liminf-H-t}\\
	\int_{Q_T} |\HHH|^2\,d\mu \,&\leq\, \liminf_{l\to\infty}  \int_0^T\int_{M^l_t} |\HHH_l(\cdot,t)|^2\,d\Ha^n\, dt. \label{eq:liminf-H}
\end{align}
\end{prop}
\begin{proof}
By \eqref{eq:class-S-l} and Fatous Lemma we have
\begin{gather*}
	h(t)\,:=\, \liminf_{l\to\infty} \int_{\R^{n+1}} |\HHH_l(\cdot,t)|^2\,d\Ha^n \,\in\, L^1(0,T)
\end{gather*}
and in particular $h(t)<\infty$ for almost every $t\in (0,T)$.  We fix such $t\in (0,T)$ and deduce from Allards compactness Theorem \cite{Alla72} that there exists a subsequence $l'\to\infty$ and an integral varifold $\tilde{\mu}_t$ with weak mean curvature $\HHH(\cdot,t)\in L^2(\tilde{\mu_t})$ such $\mu_t^{l'}\to\tilde{\mu}_t$ as varifolds, and such that
\begin{gather}
	\int |\HHH(\cdot,t)|^2 \,d\tilde{\mu}_t\,\leq\, h(t)\,=\, \liminf_{l\to\infty} \int_{\R^{n+1}} |\HHH_l(\cdot,t)|^2\,d\Ha^n. \label{eq:liminf-H-t1}
\end{gather}
From \eqref{eq:conv-mu-t} we deduce that $\mu_t=\tilde{\mu}_t$. In particular, $\mu_t$ is an integral varifold with weak mean curvature in $L^2(\mu_t)$ which satisfies \eqref{eq:liminf-H-t}. The estimate \eqref{eq:liminf-H} follows from \eqref{eq:liminf-H-t} and Fatous Lemma. Since an integral varifold is uniquely determined by the mass measure, we see that the whole sequence $l\to\infty$ from \eqref{eq:conv-mu-t} converges to $\mu_t$ in the varifold topology. This shows \eqref{eq:conv-mu-t-var}.

\iflong
It remains to prove \eqref{eq:mu-H-conv}. As above, we see that $\mu_t^{l'}\to\mu_t$ as varifolds for any subsequence $l'\to\infty$ with
\begin{gather*}
	\limsup_{l'\to\infty} \int_{M^{l'}_t} |\HHH_{l'}(\cdot,t)|^2\,d\Ha^n\,<\,\infty.
\end{gather*}
For $s,l\in\N$  we next set
\begin{equation}
	B_{l,s} := \Big\{ t \in [0,T] \,\, : \,\, \int_{\R^{n+1}} \HHH^2_l \rd \mu^l_t >s\Big\}
\end{equation}
and observe by \eqref{eq:class-S-l} that
\begin{equation} \label{BoundBadSet}
 	\Lambda > \int_{Q_T} \HHH^2_l \rd \mu^l_t\, \rd t\, >\, \big| B_{l,s} \big| s.
\end{equation}
Let us denote with $B^c_{l,s}$ the complement of $B_{l,s}$ in $[0,T]$ and
for any $\xi \in C^0_c(\R^{n+1})$ define
\begin{equation}
	T^t_{l,s}(\xi) := \begin{cases}
                            - \int_{\R^{n+1}} \HHH_l (\cdot,t) \cdot\xi \,\rd \mu_t^l& \textrm{for}\, t\in B^c_{l,s} \\
                            - \int_{\R^{n+1}} \HHH (\cdot,t) \cdot\xi \,\rd \mu_t & \textrm{for}\, t\in B_{l,s}
                          \end{cases}
\end{equation}
Under our assumptions, it is now clear that for any $ \eta \in
C^0_c(Q_T, \R^{n+1})$ we have
\begin{equation}
	T^t_{l,s} (\eta(\cdot , t)) \rightarrow - \int_{\R^{n+1}} \eta(\cdot,t)\cdot \HHH(\cdot,t)
	\rd \mu_t \qquad \textrm{as} \,\, l\rightarrow \infty 
\end{equation}
and that the following estimate holds
\begin{equation}
	T^t_{l,s} (\eta(\cdot , t))\leq ||\eta||_{C^0(\R^{n+2}_{0,T})} \sqrt{s} +
	\int_{\R^{n+1}} |\eta(\cdot,t)| |\HHH(\cdot,t)| \rd \mu_t,
\end{equation}
where the right hand side is uniformly bounded in $L^2(0,T)$ with respect to
$l$. Thus, by Lebesgue Dominated Convergence Theorem, we have that
\begin{equation} \label{PartialConvergence}
 	\int^T_0 T^t_{l,s} (\eta(\cdot , t)) \rd t \rightarrow - \int_{Q_T} \eta(\cdot,t) \cdot \HHH(\cdot,t) \rd \mu_t \rd t \qquad \quad\text{ when }\quad l
	\rightarrow \infty.
\end{equation}
We set now
\begin{equation}
	R_l := \Big| \int_{Q_T} \HHH_l (\cdot,t) \cdot\xi \,\rd \mu_t^l \rd t - \int_{Q_T}  \HHH(x,t) \eta(\cdot,t) \rd \mu_t \rd t \Big| \,.
\end{equation}
Using \eqref{BoundBadSet} and \eqref{PartialConvergence}, we estimate 
\begin{align*}
	\limsup_{l \rightarrow + \infty} R_l \,& \leq\,  \limsup_{l \rightarrow \infty} \Big[ \Big| \int^T_0  T^t_{l,s}(\eta(\cdot , t)) \rd t +\int_{Q_T} \HHH(x,t) \eta(x,t) \rd \mu_t \rd t \Big| \\
          & \qquad\qquad + \Big| \int_{B_{l,s}} \int_{\R^{n+1}}  \HHH(\cdot,t) \eta(\cdot,t) \rd
		\mu_t \rd t \Big|  + \Big| \int_{B_{l,s}} \int_{\R^{n+1}} \HHH_l (\cdot,t) \cdot\xi \,\rd \mu_t^l \rd t\Big|\Big]\\
            & \leq \,  ||\eta||_{C^0(Q_T)} \limsup_{l \rightarrow \infty}|B_{l,s}|^{\frac{1}{2}} \big(\sup_{l,t}\mu^l_t(\R^{n+1}) + \sup_{t}\mu_t(\R^{n+1})\big)^{\frac{1}{2}} \cdot \\
            &\qquad\qquad \cdot \Big(\limsup_{l\to\infty} \|\HHH_l\|_{L^2(\mu^l)}+ \|\HHH\|_{L^2(\mu)}\Big)\\
            &\leq\, C(\Lambda,T,\Omega(0)) ||\eta||_{C^0(Q_T)}\frac{1}{\sqrt{s}},
\end{align*}
where  in the last line we have also used \eqref{eq:class-S-l}, the estimate \eqref{eq:area-1}, and \eqref{eq:liminf-H}. Since $s > 0$ was arbitrary, we obtain \eqref{eq:mu-H-conv}.
\fi
\end{proof}
We obtain in the next step that the limit evolution has a generalized velocity.
\begin{prop}\label{prop:velo}\cite[Thm. 4.4]{MuRoe09}
There exists a subsequence $l\to\infty$ and a function $v\in L^2(\mu;\R^{n+1})$ such that $(\mu^l,v_l)\,\to\, (\mu,v)$ as measure-function pairs, i.e.
\begin{align}
	\lim_{l\to\infty} \int_0^T\int_{M^l_t} v_l(\cdot,t)\cdot \eta(\cdot,t)\,d\Ha^n\,dt \,&=\, \int_{Q_T}v(\cdot,t)\cdot \eta(\cdot,t)\,d\mu_t\,dt. \label{eq:conv-v}
\end{align}
We moreover have the estimate
\begin{align}
	\int_{Q_T}|v(\cdot,t)|^2\,d\mu_t\,dt \,&\leq\, \liminf_{h\to 0}\int_0^T\int_{M^l_t} |v_l(\cdot,t)|^2 \,d\Ha^n\,dt. \label{eq:liminf-v}
\end{align}
Finally $v$ is the generalized speed of the evolution $(\mu_t)_{t\in(0,T)}$ in the sense of Definition \ref{def:L2-flow}.
\end{prop}
\begin{proof}
From \eqref{eq:ass-cpct3} for $\SSSigma_l$, the convergence \eqref{eq:conv-mu} and the compactness and lower semicontinuity property for measure-function pairs \cite[Theorem 4.4.2]{Hutc86}, we conclude the existence of a subsequence $l\to\infty$ and a limit $v\in L^2(\mu;\R^{n+1})$ with \eqref{eq:conv-v} and \eqref{eq:liminf-v}. 

By \eqref{eq:gen-velo-S-l}, \eqref{eq:ass-cpct3} for $\SSSigma_l$, \eqref{eq:conv-v}, and \eqref{eq:liminf-v} we deduce that for any $\eta\in C^1_c(Q_T)$ with $|\eta|\leq 1$
\begin{align*}
	&\Big|\int_{Q_T} ( \partial_t \eta(\cdot,x) + \nabla\eta(\cdot,x)\cdot v(\cdot,t) ) \,d\mu_t\,dt\Big| \\
	\leq\,& \liminf_{l\to\infty} \Big|\int_{Q_T} ( \partial_t \eta(\cdot,x) + \nabla\eta(\cdot,x)\cdot v_l(\cdot,t) ) \,d\mu^l_t\,dt\Big| \,
	\leq\,  \frac{\Lambda}{2}.
\end{align*}
We therefore deduce \eqref{eq:gen-velo}. It remains to show that $v(x,t)$ is normal to $T_x\mu_t$ for $\mu-$almost all $(x,t)\in Q_T$.
\iflong
The proof is adapted from \cite[Proposition 3.2]{Mose01}, see also \cite[Lemma 6.3]{MuRoe09}.

We associate to $\mu^l, \mu$ the Radon measures $V_l,V \in C^0_c(Q_T\times\R^{(n+1)\times (n+1)})^* $ defined by
\begin{align}
  {V}_l(f)\,&:=\, \int_{\R^{n+1}\times(0,T)}
  f(x,t,P_l(x,t))\,d{\mu}^l_t(x), \label{eq:def-V-eps} \\
  V(f)\,&:=\, \int_{\R^{n+1}\times(0,T)}
  f(x,t,P(x,t))\,d\mu_t(x) \label{eq:def-V} 
\end{align}
for $f\in C^0_c(Q_T)\times\R^{(n+1)\times (n+1)})$, where $P_l(x,t), P(x,t)$ denote the projection onto $T_x\mu^l_t$ and $T_x\mu_t$, respectively.

From \eqref{eq:conv-mu-t-var} and Lebesgue's Dominated Convergence
Theorem we deduce that
\begin{gather}
  \lim_{l\to\infty} {V}_l \,=\, V \label{eq:conv-V}
\end{gather}
as Radon-measures on $Q_T\times\R^{(n+1)\times (n+1)}$.

Next we define functions $\hat{v}_l$ on $\spt(\mu^l)\times\R^{(n+1)\times (n+1)}$ 
by
\begin{gather*}
  \hat{v}_l(x,t,Y)\,=\, v_l(x,t)\quad\text{ for all
  }(x,t)\in\spt(\mu^l),\, Y\in\R^{(n+1)\times (n+1)}.
\end{gather*}
We then observe that
\begin{gather*}
  \int_{Q_T\times\R^{(n+1)\times (n+1)}} |\hat{v}_l|^2\,dV_l\,=\,
  \int_{Q_T} |v_l|^2\,d\mu^l\,\leq\, \frac{\Lambda}{2}
\end{gather*}
and deduce from \eqref{eq:conv-V} and \cite{Hutc86} the
existence of $\hat{v}\in L^2(V,\R^{n+1})$ such that $(V_l,\hat{v}_l)$
converge to $(V,\hat{v})$ as measure-function pairs on
$Q_T\times\R^{(n+1)\times (n+1)}$ with values in $\R^{n+1}$.

We consider now $h\in
C^0_c(\R^{(n+1)\times (n+1)})$ such that $h(Y)=1$ for all projections $Y$.
We deduce 
that for any $\eta\in C^0_c(Q_T,\R^{n+1})$
\begin{align*}
  \int_{Q_T} \eta\cdot v\,d\mu\,&=\, \lim_{l\to\infty}
  \int_{Q_T\times\R^{(n+1)\times (n+1)}} \eta(x,t)\cdot h(Y)\hat{v}_l(x,t,Y)\,
  dV_l(x,t,Y)\\ 
  &= \int_{Q_T}
  \eta(x,t)\cdot\hat{v}(x,t,P(x,t))\,d\mu(x,t),
\end{align*}
which shows that for $\mu$-almost all $(x,t)\in Q_T$
\begin{gather}
  \hat{v}(x,t,P(x,t))\,=\, v(x,t). \label{eq:hat-v}
\end{gather}
Finally we observe that for $h,\eta$ as above
\begin{align*}
  &\int_{Q_T} \eta(x,t)\cdot P(x,t)v(x,t)\,d\mu(x,t)\\
  =\,&\int_{Q_T\times\R^{(n+1)\times (n+1)}} \eta(x,t)h(Y)\cdot Y
  \hat{v}(x,t,Y)\,dV(x,t,Y)\\ 
  =\,&\lim_{l\to\infty}
  \int_{Q_T\times\R^{(n+1)\times (n+1)}}
  \eta(x,t)h(Y)\cdot Y\hat{v}_l(x,t,Y)\,dV_l(x,t,Y)\\ 
  =\,&\lim_{l\to\infty}\int_{Q_T} \eta(x,t)\cdot
  P_l(x,t)v_l(x,t)\,d{\mu}^l_t(x)\,=\, 0
\end{align*}
since $P_l v_l=0$. This shows that $P(x,t)v(x,t)=0$ for
$\mu$-almost all $(x,t)\in Q_T$.
\fi
\end{proof}
Before we show that the limit evolution of phases satisfies \eqref{eq:ass-gen-evol12} we need some preparations. First we define
for $r>0$, $(x_0,t_0)\in Q_T$ the cylinders
\begin{gather*}
  Q_r(t_0,x_0)\,:=\,B^{n+1}(x_0,r) \times (t_0-r,t_0+r).
\end{gather*}
\begin{Proposition}\label{prop:upp-dens-mu}\cite[Prop 8.1]{MuRoe09}
The measure $\mu$ is absolutely continuous with respect to $\Ha^{n+1}$,
\begin{gather}
  \mu\,\ll \,\Ha^{n+1}. \label{eq:lem:ac}
\end{gather}
\end{Proposition}
\begin{proof}
\iflong
For $t_0\in (0,T)$, $x_0\in \R^{n+1}$
we obtain from the monotonicity formula \cite[(A.6)]{KuSc04} that for any $0<r<r_0<\min\{t_0,T-t_0\}$
\begin{align}
  &\frac{1}{r}\int_{t_0-r}^{t_0+r}
  r^{-n}\mu_t\big(B(x_0,r)\big)\,dt \notag\\
  \leq\,
  &\frac{2}{r}\int_{t_0-r}^{t_0+r}
  r_0^{-n}\mu_t\big(B_{r_0}^{n}(x_0)\big)\,dt
  + C\frac{1}{r}\int_{t_0-r}^{t_0+r}  \int_{\R^{n+1}} |H(\cdot,t)|^2\,d\mu_t\,dt. \label{eq:upper-int}
\end{align}
Using \eqref{eq:liminf-H} we deduce that
\begin{gather}
  t\,\mapsto\, \int_{\R^{n+1}} |H(\cdot,t)|^2\,d\mu_t\quad\text{ is in
  }L^1(0,T), \label{eq:l1-fatou}
\end{gather}
and that for almost all $t_0\in (0,T)$ 
\begin{align*}
  &\limsup_{r\searrow 0}\frac{1}{r}\int_{t_0-r}^{t_0+r}
  r^{-n}\mu_t\big(B^{n+1}(x_0,r)\big)\,dt \\
  \leq\,& 
  2r_0^{-n}C(\Lambda,T,\Omega(0))  + \int_{\R^{n+1}} |H(\cdot,t)|^2\,d\mu_{t_0}\,<\, \infty.
\end{align*}
Since the right-hand side is
finite for $\LL^1$-almost all $t_0\in (0,T)$, this implies that
$\theta^{*(n+1)}(\mu,(x_0,t_0))$ is bounded for almost all $t_0\in (0,T)$ and all
$x_0\in\R^{n+1}$, in particular together with \eqref{eq:area-1-lim} we deduce 
\begin{align}
	\limsup_{r\searrow 0} r^{-(n+1)}\mu(B^{(n+2)}((x_0,t_0),r)) \,<\, \infty \label{eq:upp-dens-mu-2}
\end{align}
for $\mu$-almost all $(x_0,t_0)$.

Finally let $B\subset Q_T$ be given with
\begin{gather}
  \Ha^{n+1}(B)\,=\, 0. \label{eq:ass-nullset}
\end{gather}
Consider the family of sets $(D_k)_{k\in\N}$,
\begin{gather*}
  D_k\,:=\, \{z\in \Omega_T\,:\, \theta^{*(n+1)}(\mu,z)
   \leq k\}.
\end{gather*}
By \eqref{eq:upp-dens-mu-2}, \cite[Theorem 3.2]{Simo83}, and
\eqref{eq:ass-nullset} we 
obtain that for all $k\in\N$ 
\begin{gather}
  \mu(B\cap D_k)\,\leq\, 2^{(n+1)}k\Ha^{n+1}(B\cap D_k)\,=\, 0. \label{eq:null-k}
\end{gather}
Moreover, we have that
\begin{align}
  \mu(B\setminus\bigcup_{k\in\N}D_k)\,&=\,0
  \label{eq:null-infty}
\end{align}
by \eqref{eq:upp-dens-mu-2}. By \eqref{eq:null-k},
\eqref{eq:null-infty} we conclude that
\begin{gather*}
  \mu(B)\,=\, 0,
\end{gather*}
which proves \eqref{eq:lem:ac}.
\fi
\end{proof}
We need to show that the generalized tangent plane of $\mu$
exists $\Ha^{n+1}$-almost
everywhere on $\partial^*\{u=1\}$. We first obtain the following relation between the
measures $\mu$ and $|\nablatx u|$.
\begin{Proposition}\label{prop:lower-bound-mu}\cite[Prop. 8.2]{MuRoe09}
For the total variation measure $|\nablatx u|$ we have
\begin{align}
	|\nablatx u|\,\leq\, g\mu, \label{eq:abscont}
\end{align}
for a function $g\in L^2(\mu)$. In particular, $|\nablatx u|$ is absolutely continuous with respect to
$\mu$.

Moreover, the tangent plane to $\mu$ exists at $\Ha^{n+1}$-almost-all points of $\partial^*\{u=1\}$.
\end{Proposition}
\begin{proof}
By \eqref{eq:ass-gen-evol12}, \eqref{eq:ass-gen-evol4}, and \eqref{eq:cpct-BV-1} we deduce that for any $\eta\in C^1_c(Q_T)$ with $|\eta|\leq 1$
\begin{align}
	&\Big| \int_{Q_T} -\partial_t\eta u \,d\LL^{n+2}\Big| \\
	=\,& \Big|\lim_{l\to\infty} \int\limits_{Q_T} -\partial_t\eta u_l \,d\LL^{n+2} \Big|\
	\leq\, \liminf_{l\to\infty} \int_{Q_T} |\eta|(\cdot,t) |v_l(\cdot,t)|\,d\mu^l_t \,dt.
	\label{eq:est-vl-nu}
\end{align}
By \eqref{eq:class-S-l} and \cite[Theorem 4.4.2]{Hutc86}, there exists a subsequence $l\to\infty$ and $\tilde{g}\in L^2(\mu)$, $\tilde{g}\geq 0$ such that $(\mu^l,|v_l|)\to (\mu,\tilde{g})$ as $l\to\infty$ and such that
\begin{align*}
	\int_{Q_T}  \tilde{g}^2 \,d\mu \,\leq\, \int_{Q_T}  |v_l|^2 \,d\mu^l\,\leq\, \Lambda.
\end{align*}
By \eqref{eq:est-vl-nu} we therefore get
\begin{align*}
	\Big| \int_{Q_T} -\partial_t\eta u \,d\LL^{n+2}\Big|\,\leq\,  \int_{Q_T} |\eta| \tilde{g}\,d\mu,
\end{align*}
which shows that
\begin{align}
	|\partial_t u| \,\leq\, \tilde{g} \mu. \label{eq:abscont-dtu}
\end{align}
Similarly, we find
\begin{align*}
	\Big| \int_{Q_T} -\nabla\eta u \,d\LL^{n+2}\Big| 
	=\,& \Big|\lim_{l\to\infty} \int_{Q_T} -\nabla\eta  u_l \,d\LL^{n+2} \Big| \\
	=\,& \lim_{l\to\infty} \Big|\int_{Q_T} \eta \nu_l |\nabla u_l | \Big| \\
	\leq\,&\liminf_{l\to\infty} \int_{Q_T} |\eta|(\cdot,t) \,d\mu^l_t \,dt\,
	=\, \int_{Q_T} |\eta|(\cdot,t) \,d\mu,
\end{align*}
which yields $|\nabla u|\,\leq\, \mu$. Together with \eqref{eq:abscont-dtu}, we obtain \eqref{eq:abscont} and deduce that $|\nablatx u|$ is absolutely continuos with respect to $\mu$.

The final statement has been proved in Proposition \cite[Proposition 8.3]{MuRoe09}.
\end{proof}
\begin{prop}\label{prop:velo-phases}
For the limit phase function $u$ in \eqref{eq:cpct-BV-1} the equation \eqref{eq:ass-gen-evol12} holds.
\end{prop}
\begin{proof}
We first observe  that $v\in L^1(|\nabla u|)$ since by \eqref{eq:liminf-v}, \eqref{eq:abscont}, and Proposition \ref{prop:lower-bound-mu} 
\begin{align*}
  \int_{Q_T} |v|\,d|\nabla u| \,\leq\, \int_{Q_T}
  |v|\,d|\nablatx u| \,&\leq\, \int_{Q_T} g|v|\,d\mu\\
  \,&\leq\, \|g\|_{L^2(\mu)}\|v\|_{L^2(\mu)}\,<\, \infty.
\end{align*}
For $v \in L^2(\mu; \R^{n+1})$ there exist a sequence $\e \rightarrow 0$, with $v_{\e} \in C^0_c(\R^{n+1} ; \R^{n+1})$ and $v_{\e} \rightarrow v$ in $L^2(\mu , \R^{n+1})$. By \cite[Proposition 3.3]{MuRoe09} we know that for $\mu$-almost all $(x,t)\in Q_T$ at which the tangential plane of $\mu$ exists
the vector
\begin{align}
  \begin{pmatrix}
    v(x,t)\\1
  \end{pmatrix}
  \,\in\, \R^{n+1}\times\R \qquad\text{ is perpendicular to }T_{(x,t)}\mu. \label{eq:perp}
\end{align}
By Proposition \ref{prop:lower-bound-mu}, this implies
\begin{equation}
 \begin{pmatrix}
    1 \\ v
  \end{pmatrix} \cdot \nu' = 0 \qquad |\grad' u|-\textrm{a.e.} \,,
\end{equation}
where $\nu'$ denotes the generalized inner normal of $\{u=1\}$ on $\partial^*\{u=1\}$. It follows that
\begin{equation}
	\int_{Q_T} \eta \begin{pmatrix} 1 \\ v \end{pmatrix} \cdot \nu' d |\grad' u| dt = 0 \,,
\end{equation}
hence
\begin{equation}
\begin{split}
\Big| \int_{Q_T} \eta \begin{pmatrix} 1 \\ v_{\e} \end{pmatrix} \cdot \nu' d |\grad' u| \Big| & = \Big| \int_{Q_T} \eta \begin{pmatrix} 0 \\ v_{\e} - v \end{pmatrix} \cdot \nu' d |\grad' u| \Big| \\
                                                                                                                                                                           & \leq ||\eta||_{C^0(Q_T)} \int_{Q_T} \eta |v_{\e} - v| d\mu \rightarrow  0 \quad \textrm{for} \,\, \e \rightarrow 0\,.
\end{split}
\end{equation}
Therefore it holds
\begin{equation}
\begin{split}
0 & = \lim_{\e \rightarrow 0} \int_{Q_T} \eta \begin{pmatrix} 1 \\ v_{\e} \end{pmatrix} \cdot \nu' d |\grad' u| = - \lim_{\e \rightarrow 0} \int_{Q_T} (\partial_t \eta u + \grad \cdot (\eta v_{\e}) u) dx dt \\
   & = \lim_{\e \rightarrow 0} \int_{Q_T} (- \partial_t \eta u + \eta v_{\e} \cdot \grad u) dx dt = -\int_{Q_T} \partial_t \eta u dx dt + \int_{Q_T} \eta v \cdot \nu d|\grad u| dt
\end{split}
\end{equation}
which proves \eqref{eq:ass-gen-evol12}.
\end{proof}
%
We are now in the position to complete the proof of Theorem \ref{thm:main}, in particular the lower-semicontinuity of the action functional.
\begin{proof}[Proof of Theorem \ref{thm:main}]
The compactness statements have already been proved above. The $L^2$-flow property has been shown in Proposition \ref{prop:cpct-bdry}, Proposition \ref{prop:mu-t}, and Proposition \ref{prop:velo}. The assertions \eqref{eq:ass-gen-evol3}, \eqref{eq:ass-cpct1} have been proved in Proposition \ref{prop:cpct-phases}, the property \eqref{eq:ass-gen-evol4} in Proposition \ref{prop:cpct-bdry}, and \eqref{eq:ass-gen-evol12} in Proposition \ref{prop:velo-phases}. It therefore remains to show the lower-semicontinuity statements.

By \eqref{eq:mu-H-conv}, \eqref{eq:conv-v} we deduce the measure-function-pair weak convergences
\begin{align*}
	(\mu^l,v_l-H_l) \,&\to\, (\mu,v-H),\qquad (\mu^l,v_l+H_l) \,\to\, (\mu,v+H)\qquad\text{ as }l\to\infty.
\end{align*}
The lower-semicontinuity statement \cite[Theorem 4.4.2]{Hutc86} implies that for any $\tilde{\eta}\in C^0(\R^{n+1}\times [0,T])$ with $\tilde{\eta}\geq 0$
\begin{align}
	\int_{Q_T} \tilde{\eta}|v-H|^2\,d\mu \,\leq\, \liminf_{l\to\infty} \int_{Q_T} \tilde{\eta}|v_l-H_l|^2\,d\mu^l, \label{eq:liminf-v-H}\\
	\int_{Q_T} \tilde{\eta}|v+H|^2\,d\mu \,\leq\, \liminf_{l\to\infty} \int_{Q_T} \tilde{\eta}|v_l+H_l|^2\,d\mu^l. \label{eq:liminf-v+H}
\end{align}
Together with \eqref{eq:ass-cpct1} for $\SSSigma_l$ and  \eqref{eq:conv-v}, we deduce for any $\eta\in C^1(\R^{n+1}\times [0,T])$ with $0\leq \eta \leq 1$
\begin{align*}
	&2|\nabla u(\cdot,T)|(\eta(\cdot,T)) - 2 |\nabla u(\cdot,0)|(\eta(\cdot,0))  \notag\\
	&\qquad\qquad + \int_{Q_T} -2\big(\partial_t\eta + \nabla\eta\cdot v\big) +(1-2\eta)_+\frac{1}{2}|v-H|^2\,d\mu_t\,dt\\
	\leq\,& \liminf_{l\to\infty} \Big[ 2|\nabla u_l(\cdot,T)|(\eta(\cdot,T)) - 2 |\nabla u_l(\cdot,0)|(\eta(\cdot,0))  \notag\\
	&\qquad\qquad + \int_{Q_T} -2\big(\partial_t\eta + \nabla\eta\cdot v_l\big) +(1-2\eta)_+\frac{1}{2}|v_l-H_l|^2\,d\mu^l_t\,dt \Big] \\
	\leq\,& \liminf_{l\to\infty} \SSS_+(\SSSigma_l).
\end{align*}
By taking the supremum over $\eta$ we deduce
\begin{align*}
	\SSS_+(\SSSigma)\, \leq\,& \SSS_+(\SSSigma_l).
\end{align*}
Similarly we obtain $\SSS_-(\SSSigma)\, \leq\, \SSS_-(\SSSigma_l)$ and therefore \eqref{eq:lsc-SSS}. Together with the properties proved above this in particular implies \eqref{eq:ass-cpct3} and $\SSSigma\in \M(T,\Omega(0),\Omega(T))$.
\end{proof}
\section{Smooth stationary points of the action functional}
\label{sec:EL}
In the following we take into consideration smooth evolutions of smooth surfaces and characterize stationary points and conserved quantities of the action functional.
In this part it is more convenient to describe evolutions by families of embeddings. We therefore introduce the following setting.
\begin{Definition}\label{def:para}
Fix a smooth $n$-dimensional compact, orientable manifold $M$ without boundary. Let $\PPPhi : M \times [0,T] \to \R^{n+1}$ be a smoothly evolving one parameter
family of embeddings $\phi_t:=\PPPhi(\cdot,t), t\in [0,T]$. By $\SSSigma:=(\Sigma_t)_{t\in [0,T]}$, $\Sigma_t:= \phi_t(M)$ we denote the smooth evolution of snooth hypersurfaces associated to $\PPPhi$, where , in slight abuse of notation, we used the same symbols which we used in the the
preceding sections for the evolutions of the surface area measure and the
inner set.\\
The family of Riemannian measures on $M$ induced by the parametrizations
$\phi_t, t\in [0,T]$ via pullbak will be denoted with $(\bmu_t)_{t\in [0,T]}$.
Once more, in a slight abuse of notation, we denote by $\nu:M\times [0,T]\to
\R^{n+1}$ the family of inner unit normals of the sets enclosed by the
hypersurfaces $\Sigma_t$, and by $v(\cdot,t)$, $H(\cdot,t): M\to\R^{n+1}$
the scalar normal velocity and the scalar mean curvature of $\Sigma_t$ given by
\begin{align*}
	v(x,t)\,&:=\, \partial_t\PPPhi(x,t)\cdot\nu(x,t),\quad
	H(x,t)\,:=\, \vec{H}_{\Sigma_t}(\PPPhi(x,t))\cdot\nu(x,t)
\end{align*}
for $x\in M$, $t\in [0,T]$.

We say that $(\PPPhi^\eps)_{-\eps_0<\eps<\eps_0}$  is a smooth normal variation of $\PPPhi$ which preserves initial and final data, if the $\PPPhi^\eps$ are given by a smooth map $\Phi\,:\, M\times [0,T]\times (-\e_0 , \e_0)\,\to\, \R^{n+1}$ as $\PPPhi^\eps \,=\, \Phi(\cdot,\cdot,\eps)$ and if  
\begin{align*}
	&\PPPhi^0\,=\, \Id,\qquad \partial_{\e}|_{\e=0} \PPPhi^\e \,=\, f \nu ,\\
	&\PPPhi^\eps(\cdot,0)\,=\,\PPPhi(\cdot,0),\quad
\PPPhi^\eps(\cdot,T)\,=\,\PPPhi(\cdot,T)\qquad\text{ for all
}-\eps_0<\eps<\eps_0,
\end{align*}
where $f:M\times [0,T]\to\R^{n+1}$, with $f(\cdot,0)=f(\cdot,T)=0$, is smooth.
We set $\phi^\eps_t\,=\, \PPPhi^\eps(\cdot,t)\,=\,\Phi(\cdot,t,\eps)$ and denote
by $\bmu_t^\eps$, $\nu^\eps(\cdot,t)$, $t\in [0,T]$, $-\eps_0<\eps<\eps_0$,
the pullback measures and normal fields associated with $\PPPhi^\eps$, and by
$v^\eps$, $H^\eps$ the scalar velocity and scalar mean curvature fields on
$M\times [0,T]$ associated to $\PPPhi^\eps$. Finally, we call the vector field
$X := f\nu$ the \emph{variation field} associated to the given variation and set
$X_t:= X(\cdot,t)$.
\end{Definition}
Note that if $\SSSigma$ is given by a smooth evolution of smooth embeddings
$\PPPhi$ as above, the action functional $\SSS$ reduces to 
\begin{align}
	\SSS(\PPPhi)\,:=\, \SSS(\SSSigma) \,&=\,  \int_0^T\int_{M}
\Big( v^2(\cdot,t) + H^2(\cdot,t)\Big)\,d\bmu_t\,dt. \label{eq:action-para}
\end{align}

\subsection{Variation Formulae}\label{sec:var}

In this section we make some preliminary computations which will be needed for
the deduction of the smooth Euler-Lagrange equation for the functional $\SSS$.
For the notation and the fundamental identities from differential geometry
we refer to Appendix \ref{AppDG}. 
\begin{Lemma} The following variation formulae hold:
\begin{equation} \label{VariationMeasure}
	\partial_{\e}|_{\e=0} \rd \bmu_t^\eps \,=\, - \HHH \op X , \nu \cl \rd \bmu_t  \,=\, - \HHH f \rd \bmu_t\,,
\end{equation}
\begin{equation}\label{VariationNu}
	 \partial_{\e}|_{\e =0} \nu^\eps = - \grad f \,,
\end{equation}
\begin{equation} \label{VariationH}
	\partial_{\e}|_{\e=0} \HHH^\eps  = \D f + f |\AAA|^2 \,.
\end{equation}
\end{Lemma}
\begin{proof}
If we denote with $g$ and $g^\eps$ the Riemannian metrics induced respectively
by the embeddings $\PPPhi$ and $\PPPhi^\eps$
in $\R^{n+1}$, we have
\begin{equation*}
	\begin{split}
		\partialeo g_{ij}^\eps & = \partial_{\e}|_{\e=0} \op  \partial_i \phi_t^\eps ,
		\partial_j \phi_t^\eps\cl = \partial_i \op X , \partial_j \phi_t \cl + \partial_j \op
		X , \partial_i \phi_t \cl - 2 \op X , \partial^2_{ij} \phi_t \cl \\
                                           & = \partial_i \op X , \partial_j
		\phi_t \cl + \partial_j \op X , \partial_i \phi_t \cl - 2 \G^r_{ij} \op X ,
		\partial_r \PPPhi \cl - 2 \hhh_{ij} \op X , \nu \cl \\
                                           & = - 2 \hhh_{ij} \op X , \nu \cl = - 2 f \hhh_{ij}.
	\end{split}
\end{equation*}
Since by definition $g_{ij}^\eps (g^{\eps})^{jk} = \d^k_i$, one gets
\begin{equation*}
	\partial_{\e}|_{\e=0} (g^{\eps})^{ij} = 2 f \hhh^{ij} \,.
\end{equation*}
Using the formula $\partial_{\e} \det (A_{\e}) = \det (A_{\e}) \tr
[A^{-1}_{\e} \partial_{\e} A_{\e}]$ we obtain the following equation which
describes the variation of the induced Riemannian measure
\begin{equation*}
	\begin{split}
		\partial_{\e}|_{\e=0} \rd \bmu^{\e}_t & := \partial_{\e}|_{\e=0} \sqrt{\det (g^\eps)} =
		\frac{\sqrt{\det (g)} g^{ij} \partial_{\e}|_{\e=0}g_{ij}^\eps}{2} \\
                                            & = \frac{ - 2 \sqrt{\det (g)}g^{ij}
		h_{ij}}{2}  \op X , \nu \cl\\
                                            &  = - \HHH f  \rd \bmu_t \,. \\                                            
	\end{split}
\end{equation*}
For the variation of the normal vector to the hypersurface we get
\begin{equation*}
	\op \partial_{\e}|_{\e=0} \nu^\eps , \partial_i \phi_t^\eps \cl  = - \op \nu , \partial_i 
	\partial_{\e}|_{\e=0} \phi_t^\eps \cl = - \op \nu , \partial_i (f \nu) \cl = - \partial_i
	f,
\end{equation*}
which means,
\begin{equation*}
	\partial_{\e}|_{\e =0} \nu^\eps = - \grad f \,.
\end{equation*}
In order to compute the variation of the mean curvature, we start computing
the variation of the second fundamental form
\be \label{VarSecondFF}
	\partial_{\e}|_{\e =0} \hhh_{ij}^\eps = \partial_{\e}|_{\e=0} \op \nu^\eps ,
	\partial^2_{ij} \phi_t ^\eps\cl = - \op \grad f , \partial^2_{ij} \phi_t \cl +
	\op \nu , \partial^2_{ij} X \cl\,.
\ee
By (\ref{GaussWein}) and (\ref{SecondFF}) we obtain
\begin{align} \label{VarSecondFFa}
	\op \grad f , \partial^2_{ij} \phi_t \cl \,&=\, \op \grad f, \G^r_{ij}
	\partial_r \PPPhi + \hhh_{ij} \nu \cl \,=\, \grad_r f \G^r_{ij}, \\
	\op \nu , \partial^2_{ij} X \cl  \,&=\, \op \nu , \partial^2_{ij} (f \nu) \cl \,=\,
		\partial^2_{ij} f + f \op \nu , \partial^2_{ij} \nu \cl \notag\\
                                           & \,=\, \partial^2_{ij} f - f \op \nu ,
		\partial_i (\hhh_{jr} g^{rp} \partial_p \phi_t ) \cl  \notag\\
                                           & \,=\,  \partial^2_{ij} f - f \hhh_{jr}
		g^{rp}\hhh_{pi}. \label{VarSecondFFb}
\end{align}
From equations (\ref{VarSecondFF}) - (\ref{VarSecondFFb}) we finally deduce that
\begin{equation*}
	\partial_{\e} |_{\e=0} \hhh_{ij}^\eps  = \grad^2_{ij} f - f \hhh_{jr} g^{rp}
	\hhh_{pi}.
\end{equation*}
It then follows that for the variation of the mean curvature we have
\begin{equation*}
	\partial_{\e}|_{\e=0} \HHH^\eps  = \partial_{\e}|_{\e=0} (g^{ij})^\eps \hhh_{ij} + g^{ij}
	\partial_{\e}|_{\e=0} \hhh_{ij}^\eps = \D f + f |\AAA|^2 \,.
\end{equation*}
\end{proof}

\subsection{The first Variation of $\SSS$}
We have now all the tools to compute the Euler--Lagrange equation for $\SSS$.
\begin{teo}
Let $\PPPhi$ and $(\PPPhi^\eps)_{-\eps_0<\eps<\eps_0}$ be a smooth evolution of smooth embeddings and a normal variation given by a field $f$ as in Definition \ref{def:para}. Then the first variation of $\SSS$ at $\PPPhi$ in direction of $f$ is given by
\begin{equation} \label{FirstSmoothVar}
\begin{split}
	\delta\SSS(\PPPhi)(f)\,=\,\totaleo \SSS (\PPPhi^\eps) & = \int^T_0 \int_M f \, \Big[ -
	\partial_t v + \D \HHH + \HHH |\AAA|^2 - \frac{\HHH^3}{2} + \frac{v^2
	\HHH}{2}\Big] \rd \bmu_t \rd t.
\end{split}
\end{equation}
Consequently, the Euler-Lagrange equation for a smooth stationary point $\PPPhi$ of $\SSS$ is given by
\begin{equation} \label{EulerLagrange}
	\partial_t v = \D \HHH + \HHH |\AAA|^2 - \frac{\HHH^3}{2} + \frac{v^2 \HHH}{2}.
\end{equation}
\end{teo}
\begin{proof}
We start by computing the variation of the normal speed.
\begin{align} \label{VariationV}
	\partialeo v^\eps & = \op \partial_t \partialeo \phi_t^\eps ,
	\nu \cl + \op \partial_t \phi_t , \partialeo\nu^\eps \cl 
	                           = \op \partial_t ( f\nu ) , \nu \cl = \partial_t f.
\end{align}
Using equations \eqref{VariationMeasure}, \eqref{VariationH} and
\eqref{VariationV}, we can now compute
\begin{equation} \label{FirstStepVar}
	\totaleo S(\PPPhi^\eps)  = \int^T_0 \int_M
	\Big[\partial_t f v + \HHH \D f + \HHH f |\AAA|^2 - (v^2 +
	\HHH^2)\frac{f\HHH}{2}\Big] d \bmu_t  dt \,.
\end{equation}
Observing that
\begin{equation*}
	\frac{d}{dt} \int_M f v d\bmu_t = \int_M \Big[ \partial_t f v + f
	\partial_t v - f v^2 \HHH \Big] d\bmu_t \,,
\end{equation*}
we get
\begin{equation} \label{PartialInt1}
	\int^T_0 \int_M \partial_t f v \, d\bmu_t d t = \int^T_0 \int_M
	f \Big[- \partial_t v + \HHH v^2 \Big] d \bmu_t d t \,.
\end{equation}
Substituting \eqref{PartialInt1} into \eqref{FirstStepVar}, we obtain
\begin{equation*}
	\begin{split}
		\totaleo \SSS(\PPPhi^\eps) & = \int^T_0 \int_M f \, \Big[ -
		\partial_t v + \D \HHH + \HHH |\AAA|^2 - \frac{\HHH^3}{2} + \frac{v^2
		\HHH}{2}\Big] d \bmu_t d t \,,
	\end{split}
\end{equation*}
which concludes the proof.
\end{proof}
%
\section{Symmetries and Conserved Quantities}
\label{sec:symm}
In this section we will analyze some particular variations, in order to describe some properties of stationary points which are often less obvious from the Euler--Lagrange equation. In particular, we characterize some conserved  quantities along smooth evolutions which are stationary points of the action functional. 
\subsection{Energy Conservation}
The functional $\SSS$ can be formally seen as the sum of
a kinetic and a potential term depending on curvature, integrated with respect
to a time dependend measure. By analogy with Lagrangian mechanics, one can write
the formal associated Hamiltonian and can compute whether energy conservation along
stationary trajectories holds. We actually have the following property.
\begin{Proposition}
Let  $\PPPhi: M \times [0,T] \rightarrow \R^{n+1}$ be a stationary point of the Functional $\SSS$ in the class of smooth evolutions with prescribed initial and final states. Then the quantity
\begin{equation} \label{Energy}
  E(\phi_t) := \int_{M} (v^2 - \HHH^2) \, \rd \bmu_t,\quad t\in [0,T],
\end{equation}
which we will call \emph{energy}, does not depend on $t$. We will in this case use the notation $E(\PPPhi)$ for $E(\phi_t)$, $t\in [0,T]$.
\end{Proposition}
\begin{proof}
Let us consider a time reparametrization for $\PPPhi$ of the form $\PPPhi^\eps(\cdot,t)= \PPPhi(\cdot,t_\eps)$, $t_{\e}:= t + \e \eta$,
with $\eta \in C^{\infty}_0(0,T)$. One easily checks that the action functional of $\PPPhi^{\e}$ is given by
\begin{equation}
	 \SSS(\PPPhi^{\e}) =  \int^T_0\int_M ( (v^{\e})^2 + (\HHH^\eps)^2 ) d \bmu_{t}^\eps\, d t  = \int^T_0\int_M \Big( \frac{v^2}{1 + \e \eta'} + (1 + \e
	\eta')\HHH^2\Big) d \bmu_{t} d t. 
\end{equation}
For the corresponding first variation of $\SSS$ we get
\begin{equation}
	 \totaleo \SSS(\PPPhi^{\e}) = \int_0^T \eta'(t) \int_M (- v^2 +
	\HHH^2) d \bmu_t d t.
\end{equation}
Since $\PPPhi$ has been supposed to be stationary, the thesis
follows.
\end{proof}
\begin{Remark}
It is also possible to deduce Energy conservation from equations
\eqref{VariationMeasure}, \eqref{VariationH} and \eqref{VariationV}. Actually,
\begin{equation}
 \begin{split}
 	\frac{d}{dt} \int_M (v^2 - \HHH^2) \, d \bmu_t & = \int_M \Big[2 v (\D \HHH + \HHH
	|\AAA|^2 - \frac{\HHH^3}{2} + \frac{v^2\HHH}{2})- \\
                                               & \qquad\qquad - 2 \HHH (\D v + v
	|\AAA|^2) -
	(v^2 - \HHH^2) v \HHH)\Big] d \bmu_t \,
                                                =\, 0 \,.
 \end{split}
\end{equation}
\end{Remark}
\subsection{Conformal Variations}
We next investigate conformal variations of the form
\begin{equation} \label{ConfV}
	\PPPhi(x,t, \e) = e^{\aaa(t,\e)} \PPPhi(x,t),
\end{equation}
where $\aaa: [0,T] \times \R \rightarrow \R$ is a smooth function which
satisfies
\begin{align} \label{Boundarya1}
	\aaa(t,0) \,&=\, 0 \quad &&\textrm{ for all } \,t \in [0,T],\\
	\aaa(0,\eps) \,&=\, 
	\aaa(T,\eps)  \,=\, 0\quad&&\text{ for all }-\eps_0<\eps<\eps_0.\label{Boundarya}
\end{align}
We denote $\aprime(t)\,:=\,  \partial_{\e}|_{\e=0}\aaa(t,\eps)$ for $t\in [0,T]$.
The following lemma describes the variation under \eqref{ConfV} of some
geometric quantities appearing in $\SSS$.
\begin{Lemma} \label{CVarKin}
For a variation as in \eqref{ConfV} we have:
\begin{align}
	\partial_{\e} |_{\e=0} \rd \bmu^{\e}_t \,&=\, n \aprime(t) \rd \bmu_t, \label{eq:cnf-mu}\\
	\partial_{\e}|_{\e=0} v^\eps(\cdot,t) &= 
	\aprime'(t) \op \phi_t , \nu(\cdot,t) \cl + \aprime(t) v (\cdot,t), \label{eq:cnf-v}\\
	\partial_{\e} |_{\e=0} H^\eps \,&=\, -\aprime H. \label{eq:cnf-H}
\end{align}
\end{Lemma}
\begin{proof}
When the embedding undergoes a variation as in \eqref{ConfV}, the
induced metric on the corresponding embedded submanifold in $\R^{n+1}$ reads
\begin{equation} \label{ConformalVariation}
	g_{ij}^\eps (\cdot,t) \,=\, e^{2 \aaa(t,\e)} g_{ij}(\cdot,t)\,,
\end{equation}
hence
\begin{equation*}
	\partial_{\e} |_{\e=0} g_{ij}^\eps(\cdot,t) = 2 \aprime(t) g_{ij}(\cdot,t)\,,
\end{equation*}
and we conclude that the variation of the induced surface measure is given by \eqref{eq:cnf-mu}.
The normal to the hypersurface does not change along this kind of
variations, we actually have
\begin{equation*}
	0 = \partial_{\e} |_{\e=0} \op \nu^\eps (\cdot,t)  , \partial_j \phi_t ^\eps \cl = \op
	\partial_{\e}|_{\e=0} \nu^\eps (\cdot,t), \partial_j \phi_t \cl - \op \nu(\cdot,t) ,
	\partial_j (\aprime(t) \phi_t) \cl\,,
\end{equation*}
from which we get
\begin{equation*}
	\op \partial_{\e}|_{\e=0} \nu^\eps (\cdot,t) , \partial_j \phi_t \cl = 0\,.
\end{equation*}
For the normal speed, we have
\begin{equation}
	\partial_{\e}|_{\e=0} v^\eps(\cdot,t) \,=\, \op \partial_t(\aprime(t)\phi_t , \nu
	\cl =
	\aprime'(t) \op \phi_t , \nu(\cdot,t) \cl + \aprime(t) v (\cdot,t)\,.
\end{equation}
For the mean curvature we obtain that
\begin{align}
	H^\eps \,=\, e^{-\aaa(\cdot,\eps)}H \label{eq:cnf-Heps}
\end{align}
and we deduce \eqref{eq:cnf-H}.
\end{proof}
By means of Lemma \ref{CVarKin}, we are able to compute the variation of the
\emph{kinetic} term in $\SSS$ and state the following
\begin{Proposition}
When a trajectory undergoes a variation as in \eqref{ConfV}-\eqref{Boundarya}, the variations of the kinetic term, the potential term, and the full action functional $\SSS$ are given by
\begin{align}
	&\totaleo \int^T_0 \int_M (v^\eps)^2 \rd \bmu^{\e}_t\,dt \notag\\
	&\qquad =\, 2 \int^T_0
	\aprime(t) \int_M (- \partial_t v \op \PPPhi , \nu \cl + v \op \PPPhi
	, \grad v \cl + v^2 \HHH \op \PPPhi , \nu \cl + \frac{n}{2} v^2) \rd \bmu_t
	\rd t, \label{CVKinetic}
	\\
	&\totaleo \int^T_0 \int_M (\HHH^\eps)^2 \rd \bmu^{\e}_t \rd t \,
	=\, \big(
	n - 2 \big) \int^T_0 \aprime(t) \int_M \HHH^2 \rd \bmu_t \rd t, \label{CVPotential}
	\\
	&\totaleo S (\PPPhi^\eps) \notag \\
	 &\qquad =\,  2\int^T_0 \aprime(t) \int_M \Big[-
	\partial_t v \op \PPPhi , \nu \cl + v \op \PPPhi , \grad v \cl + v^2 \HHH \op
	\PPPhi , \nu \cl + \frac{n}{2} v^2 + \Big(\frac{n}{2} - 1 \Big) \HHH^2 \Big]
	\rd \bmu_t \rd t\,. \label{CVAction}
\end{align}
\end{Proposition}
\begin{proof}
From Lemma \ref{CVarKin} we get
\begin{equation}
	\totaleo \int^T_0\int_M (v^\eps)^2 \, d \bmu^{\e}_t\,dt \,=\, 2
	\int^T_0\int_M \big[v(\aprime' \op \PPPhi , \nu \cl + \aprime v) +
	\frac{n}{2} \aprime v^2\big] \, d \bmu_t \,dt.
\end{equation}
By an integration by parts in the first of the three terms in the integrand on the right-hand side, and \eqref{VariationMeasure}, \eqref{VariationNu} we obtain
\begin{align*}
		\int^T_0 \int_M v \aprime' \op \PPPhi , \nu \cl d \bmu_t d t  & =
		\int^T_0\int_M [- \partial_t v \aprime \op \PPPhi , \nu \cl  - v^2
		\aprime + v \aprime \op \PPPhi , \grad v \cl + v^2 \HHH \aprime \op \PPPhi , \nu
		\cl ] \, d \bmu_t  d t \notag\\
		& \qquad\qquad  + \Big[ \int_M v \, \aprime \op \PPPhi , \nu
		\cl d \bmu_t {\Big]}^T_0.
\end{align*}
Using equation \eqref{Boundarya} we deduce \eqref{CVKinetic}. By \eqref{ConformalVariation} and \eqref{eq:cnf-Heps} we obtain
\begin{align*}
	\totaleo\int_0^T\int_M (H^\eps)^2\,d\mu_t^\eps\,dt \,&=\, \totaleo\int_0^T\int_M H^2 e^{(-2+n)\aaa(\cdot,\eps)}\,d\mu_t\,dt\notag\\
	&=\,\int_0^T (-2+n)\aprime(t) \int_M H^2 \,d\mu_t\,dt,
\end{align*}
which gives \eqref{CVPotential}.
Together with \eqref{CVKinetic} we finally deduce \eqref{CVAction}.
\end{proof}
\begin{Remark}
We conclude that trajectories that are stationary for the
\emph{kinetic} part of the action along conformal variations as above satisfy
\begin{equation} \label{ConformalConserved1}
	\int_M (- \partial_t v \op \phi_t , \nu \cl + v \op \phi_t , \grad v
	\cl + v^2 \HHH \op \phi_t , \nu \cl + \frac{n}{2} v^2) d \bmu_t = 0 \,.
\end{equation}
On the other hand it holds
\begin{equation} \label{ConformalConserved2}
	\frac{d}{dt} \int_M v \op \phi_t , \nu \cl d \bmu_t \, =\, \int_M
	[\partial_t v \op \phi_t , \nu \cl + v^2 - v \op \phi_t , \grad v \cl - v^2 \HHH
	\op \phi_t , \nu \cl] d \bmu_t \,.
\end{equation}
Adding equations \eqref{ConformalConserved1} and
\eqref{ConformalConserved2} we obtain that for all $t\in [0,T]$
\begin{equation}
	\frac{d}{dt} \int_M v \op \phi_t , \nu \cl \rd \bmu_t = (1 + \frac{n}{2})
	\int_M v^2 d \bmu_t\,.
\end{equation}
Integrating over time, we find
\begin{equation}
	\Big[ \int_M v \op \phi_t , \nu \cl d \bmu_t \Big]^T_0 = (1 + \frac{n}{2})
	\int^T_0\int_M v^2 d \bmu_t d t \,.
\end{equation}

For $n=2$ the Willmore functional is invariant under dilations, in this case the variation of the whole action functional coincides with the variation of its \emph{kinetic} part.
\end{Remark}
We are now in the position to prove an equality which can be used to
deduce the Hamilton-Jacobi equation associated to $\SSS$.
\begin{Proposition}
For $\SSS-$stationary trajectories, the following equation holds true:
\begin{equation}
	\Big[ \int_M v \op \phi_t , \nu \cl \rd \bmu_t \Big]^T_0 = \int^T_0
	\int_M \Big[ \Big(1 + \frac{n}{2}\Big) v^2 + \Big(\frac{n}{2} - 1
	\Big) \HHH^2 \Big] \rd \bmu_t \rd t =  2 T E(\PPPhi) + n S(\PPPhi) \,.
\end{equation}
\end{Proposition}
\begin{proof}
 Since we are considering an $\SSS-$stationary trajectory, from equation
\eqref{CVAction} we have that
\begin{equation} \label{ConformalStationary}
	\int_M \Big[(- \partial_t v \op \phi_t , \nu \cl + v \op \phi_t ,
	\grad v \cl + v^2 \HHH \op \phi_t , \nu \cl + \frac{n}{2} v^2 + \Big(\frac{n}{2}
	- 1 \Big) \HHH^2 \Big] d \bmu_t = 0\,.
\end{equation}
adding (\ref{ConformalConserved2}) and (\ref{ConformalStationary}) we get
\begin{equation}
	\frac{d}{dt} \int_M v \op \phi_t , \nu \cl d \bmu_t = \int_M \Big[
	\Big(1 + \frac{n}{2}\Big) v^2 + \Big(\frac{n}{2} - 1 \Big) \HHH^2 \Big] d
	\bmu_t
	= 2
	E(\PPPhi) + \frac{n}{2} \int_M (v^2 + \HHH^2) \rd \bmu_t \,
\end{equation}
and the thesis follows integrating over time.
\end{proof}

\subsection{Isometric variations}
We now consider variations of the form 
\begin{equation} \label{IsometricV}
	\begin{split}
		 & \phi_t^\e(x) \, =\,  O(t,\e) \phi_t(x) \quad\text{ for }x\in M, t\in [0,T], -\eps_0<\eps<\eps_0, \\ 
		 & O(t,\e) \,\in\, SO(n+1), \quad O(t,0) \,=\, O(0,\e) \,=\, O(T,\e)= \Id\quad\text{ for all }-\eps_0<\eps<\eps_0, t\in [0,T].
	\end{split}
\end{equation}
It is clear that this variation leaves the area element and the mean curvature invariant. We therefore obtain the following property for the corresponding first variation.
\begin{Proposition}
If a trajectory undergoes a variation as in \eqref{IsometricV}, the first variation
of $\SSS$ reads as
\begin{equation} \label{VarIso}
	\totaleo \SSS(\PPPhi^\eps) \,=\, 2 \int_0^T \int_M \op A'(t) \phi_t(x),
	\partial_t \phi_t \cl \rd \bmu_t \rd t \,,
\end{equation}
where $A(t) = \partialeo O(\e,t)$.
\end{Proposition}
\begin{proof}
The variation of the normal speed is given by
\begin{equation*}
	v^{\e}  = \op \partial_t \phi^\eps_t , \nu^{\e}(\cdot,t) \cl \,=\, \op \partial_t\big(O(t,\e) \phi_t\big) , O(t,\e)\nu(\cdot,t)\cl.
\end{equation*}
Using that $A(t)$ is an antisymmetric matrix for any $t\in [0,T]$ this implies that
\begin{equation}
	\begin{split}
		\totaleo S(\PPPhi^\eps) & = 2 \int_0^T \int_M v(\cdot,t)\big(\op \partial_t\big(A(t) \phi_t\big),  \nu(\cdot,t) \cl + \op \partial_t \phi_t, A(t)\nu(\cdot,t)\cl\big)
		d \bmu_t d t \\
                              & = 2 \int_0^T \int_M \big(\op \partial_t\big(A(t) \phi_t\big),  \partial_t\phi_t \cl     - \op A(t) \partial_t\phi_t, \partial_t\phi_t \cl\big)        d \bmu_t d t,	\end{split}
\end{equation}
and the thesis follows.
\end{proof}
The conserved quantity along $\SSS-$stationary trajectories which arises
analyzing variations of the form \eqref{IsometricV} can be interpreted as
\emph{angular momentum}. 
\begin{cor}
Along any $\SSS-$stationary trajectory the quantity
\begin{equation}
	\int_M v (\nu(\cdot,t) \otimes \phi_t - \phi_t \otimes \nu(\cdot,t)) \rd \bmu_t \,,
\end{equation}
which can be interpreted as angular momentum, does not depend on time.
\end{cor}
\begin{proof}
The thesis follows from equation \eqref{VarIso}, choosing $A(t) = f(t)A$, with an arbitrary $f \in C^{\infty}_c(0,T)$ and noticing that
\begin{equation}
 	A : \frac{\rd}{\rd t} \int_M v(\cdot,t) \nu
	(\cdot,t) \otimes \phi_t \, d \bmu_t = 0 \,,
\end{equation}
for all the antisymmetric matrices $A$, is equivalent to the thesis.
\end{proof}
\begin{Remark}
If $\phi_0(\cdot)$ and $\phi_T(\cdot)$ are both round
spheres, it is easy to see that the constant must be zero (and that the
integrand actually vanishes pointwise on round spheres). Note also that the 
vanishing of the angular momentum does not imply that the trajectory is at every
time a round sphere, even if if the initial and final data are both round
spheres this point will be discussed further in Section \ref{sec:spher}.
\end{Remark}
%
\section{The spherical Case}
\label{sec:spher}
In this section, we will study the problem of finding optimal trajectories
connecting concentric, round $n-$spheres in $\R^{n+1}$. We
will also determine conditions under which the optimal trajectory in the class
of spherical trajectories is an
absolute minimizer of the action functional.

\subsection{Some Formulae for Graphs over Spheres}

Let $\phi: \SS^n \rightarrow \R^{n+1}$ be a smooth
embedding which can be parametrized as graph over $\SS^n$. This means that
there exists a smooth function $r: \SS^n \rightarrow \R$ such that
\begin{equation} \label{GraphSphere}
 \phi(x) = r(x) x \,\,, \,\, x \in \SS^n \,.
\end{equation}
The following equations follow by direct computations from \eqref{GraphSphere}.
\begin{prop}
If $\phi: \SS^n \rightarrow \R^{n+1}$ is a smooth embedding of $\SS^n$ into
$\R^{n+1}$, which is parametrized as a graph over the unit $n-$sphere, we have
that the naturally induced metric on $\phi (\SS^n)$ is given by
\begin{equation}
 \g_{ij} = r^2 \t_{ij} + \hgrad_i r \hgrad_j r\,,
\end{equation}
where $\t_{ij}$ is the standard metric on the unit sphere in
$\R^{n+1}$ with associated Levi-Civita connection $\hgrad$ and measure $\rd \hmu$.
The inverse of the induced metric reads
\begin{equation}
 \g^{ij} = \frac{1}{r^2} \Bigg( \t^{ij} - \frac{\hgrad^i r \hgrad^j r}{r^2 +
|\hgrad r|^2} \Bigg)\,.
\end{equation}
The inner unit normal normal vector to the embedded surface is
\begin{equation} \label{NGS}
 \nu(x) = - \frac{1}{\sqrt{r^2 + |\hgrad r|^2}} (r x - \t^{ij} \hgrad_i r
\hgrad_j r)
\,
\end{equation}
and the second fundamental form is
\begin{equation}
 \hhh_{ij} = \op \nu , \partial^2_{ij} \phi \cl = \frac{1}{\sqrt{r^2 + |\hgrad
r|^2}} (r^2 \t_{ij} + 2 \hgrad_i r \hgrad_j r - r \hgrad_i \hgrad_j r) \,,
\end{equation}
while the mean curvature can be expressed as
\begin{equation}\label{MCGS}
 \HHH = \frac{1}{r^2(r^2 + |\hgrad r|^2)^{3/2}} \Big[(n+1)r^2|\hgrad r|^2
+
n r^4 + r \hgrad_i \hgrad_j r \hgrad^i r \hgrad^j r - r \hat{\D} r (r^2 +
|\hgrad
r|^2)\Big]\,.
\end{equation}
Finally, the induced area element is given by
\begin{equation}
 \rd \bmu = r^{n-1} \sqrt{r^2 + |\hgrad r|^2} \, \rd \hat{\mu} \,.
\end{equation}
\end{prop}
\subsection{First variation around spherical trajectories}
In this section we will study the first variation of
the action functional, when restricted to the following family of trajectories.
\begin{Definition}
Given three positive real numbers $T$, $R_0$, and $R_T$, we say that a
smooth map $\phi_0 : \SS^n \times [0,T]
\rightarrow \R^{n+1}$, which for any fixed $t \in [0,T]$ is a regular
embedding of $\SS^n$ in $\R^{n+1}$, is a spherical trajectory connecting the
concentric $n-$spheres of radii $R_0$ and $R_T$, if there exists a smooth map
$r_0 : [0,T] \rightarrow \R$ such that
\begin{equation}\label{eq:def-sph-phi}
 \phi_0(x,t) = r_0 (t) x \,\,,\,\, x \in \SS^n \,,
\end{equation}
with $r_0(0) = R_0$ and $r_0(T) = R_T$.
\end{Definition}

We now compute the first variation of the action functional around an arbitrary
spherical trajectory. By the tubular neighborhood theorem, we can restrict to
variations which are graphs over spheres without any loss of generality.
\begin{lemma} \label{FirstVarSph}
Let $T$, $R_0$, $R_T$ be positive real numbers and $r_0 : [0,T]
\rightarrow \R$ a function defining a spherical trajectory $\phi_0$ as in \eqref{eq:def-sph-phi}.
Let $\rh : \SS^n \times [0,T] \rightarrow \R $ be a smooth function with
$\rh (\cdot , 0) = \rh(\cdot , T) = 0$ for any $x \in \SS^{n}$ and $\e$ a real
number. Define $r_{\e} : \SS^n \times [0,T] \rightarrow \R $ so that
$r_{\e}(x,t) = r_0(t) + \e \rh(x,t)$ and define $\phi_{\e} : \SS^n \times [0,T]
\rightarrow \R^{n+1}$ as $\phi_{\e}(x,t) = (r_0(t) + \e \rh(x,t))x$. Then it
holds
\begin{equation} \label{FVS}
\totaleo \SSS (\phi_{\e}) = - \int_0^T \int_{\SS^n} [2 \ddot{r}_0
r_0^n + n \dr_0^2 r_0^{n-1} - n^2 (n-2) r_0^{n-3}] \rh \rd \hmu \rd t \,,
\end{equation}
where the dot denotes the partial derivative with respect to $t$ and $\hmu$ is the surface measure of the standard unit $n-$sphere in $\R^{n+1}$. As a consequence, for any stationary spherical trajectory, the function $r_0$ satisfies the ordinary differential equation
\begin{equation} \label{ODEStatSph}
 2 \ddot{r}_0 r_0^n + n \dr_0^2 r_0^{n-1} - n^2 (n-2) r_0^{n-3} = 0 \,.
\end{equation}
\end{lemma}
\begin{proof}
Let us define $Q_{\e} := r^2_{\e} + |\hgrad r_{\e}|^2$. From the definition of
normal speed and \eqref{NGS} we have
\begin{equation}
\begin{split}
 v^2_{\e} := \op \partial_t \phi_{\eps} , \nu_{\e} \cl^2 & = Q^{-1}_{\e} \op
\dr_{\e} x , r_{\eps} x - \t^{ij} \hgrad_i r_{\e} \hgrad_j x \cl^2 \\
                                                & = Q^{-1}_{\e} \op
\dr_{\e} x , r_{\e} x \cl^2 = Q^{-1}_{\e} \dr^2_{\e} r^2_{\e}\,.
\end{split}
\end{equation}
Moreover, from \eqref{MCGS}  and since $\hgrad r_{\e} = \e \hgrad \rh$,we have
that
\begin{equation}
\HHH_{\e} = r^{-2}_{\e} Q^{-3/2}_{\e}[(n+1) \e^2 r_{\e}^2 |\hgrad \rh|^2 + n
r^4_{\e} + \e^3 r_{\e} \hgrad^2_{ij} \rh \hgrad_i \rh \hgrad_j \rh - \e r_{\e}
\hat{\D}\rh Q_{\e}] \,,
\end{equation}
which we rewrite as $H_{\e} = r^{-2}_{\e} Q^{-3/2}_{\e} W_{\e}$, where we have
denoted with $W_{\e}$ all the terms in between the square brackets.

For convenience, let us rewrite the action as
\begin{equation}
\SSS (\phi_{\e}) := \int_0^T \int_{\SS^n} (v^2_{\e} +
\HHH^2_{\e}) d \bmu^{\e}_t d t= \int_0^T \int_{\SS^n} (K_{\e} +
P_{\e}) d \hmu d t \,,
\end{equation}
with $K_{\e} := \dr^2_{\e} r^{n+1}_{\e} Q^{-1/2}_{\e}$ and $P_{\e} :=
r^{n-5}_{\e} Q^{-5/2}_{\e}W^2_{\e}$.
Differentiating, we get
\begin{equation} \label{SKV}
 \partial_{\e}|_{\e=0} K_{\e} = 2 \dr_0 r^n_0 \drh + (n+1)\dr^2 r^{n-1}_0 \rh
- \dr^2_0 r^{n-1}_0 \rh
\end{equation}
and
\begin{equation} \label{SPV}
\begin{split}
 \partial_{\e}|_{\e=0} P_{\e} & = n^2 (n-5) r^{n-3}_0 \rh - 5 n^2 r^{n-3}_0 \rh
+ 2 n r^{n-3}_0 (4 n \rh - \hat{\D} \rh) \\ 
                              &  = n^2 (n-2) r^{n-3}_0 \rh + 2 n
r^{n-3}_0 \hat{\D} \rh \,.
\end{split}
\end{equation}
The thesis follows by summing the two contributions and integrating by parts. 
\end{proof}
\begin{Remark} \label{GeneralCP}
Note that \eqref{ODEStatSph} is identical to the Euler--Lagrange equation of $\SSS$ restricted to the class of spherical symmetric evolutions. Moreover,  equation \eqref{ODEStatSph} follows  already from energy conservation along a spherical trajectory. In fact, for a spherical trajectory \eqref{Energy} implies
\begin{equation} \label{EnergyConsSpheres}
	\frac{d}{dt} (\dr^2_0 r^n_0 - n^2 r^{n-2}_0) = 0 \,,
\end{equation}
which is equivalent to \eqref{ODEStatSph}. Since
energy conservation is a consequence of stationarity with respect to time reparametrization, and since the class of spherical trajectories is invariant under time
reparametrization, energy conservation is clearly a necessary condition for the stationarity of spherical trajectories. Here it is also sufficient.
This also means that a critical
trajectory in the class of spherical evolutions is critical also in the larger class of
smooth trajectories. 
\end{Remark}
\begin{Remark}
In the case $n=1$, equation \eqref{EnergyConsSpheres} is equivalent to $(\dr^2_0 r_0 - r^{-1}_0) = E$, with $E \in \R$. The solutions to this equation coincide with the rotationally symmetric solutions to the $\SSS-$minimization problem given by Okabe in \cite{Okab12}. We also notice, that the three classes of solution given by Okabe correspond respectively to the cases $E < 0$, $E=0$, and $E>0$.

When $n=2$ we obtain from \eqref{EnergyConsSpheres} an explicit formula for $r_0$ (assuming without any
loss of generality that $R_0 > R_T$),
\begin{equation} \label{Optimal2DS}
	 r_0(t) = \sqrt{\frac{R_0^2 - R_T^2}{T} t + R_0^2}.
\end{equation}
In particular the unique stationary spherical solution is a time rescaled mean curvature flow.
In the following it will be convenient to compare $T$ with the time $T_{MCF}(R_0,R_T)$, needed to join 
two concentric $2-$dimensional round spheres in $\R^3$ having radii $R_0$ and
$R_T$ by (time reversed) mean curvature flow. This time is given by
\begin{gather}
	T_{MCF}\,=\,T_{MCF}(R_0,R_T) \,= \, \frac{|R_0^2 - R_T^2|}{4}. \label{eq:TMCF}
\end{gather}
\end{Remark}
\subsection{Second variation around spherical trajectories}
We now want to study the character of the stationary spherical trajectories, in
particular we want to determine conditions under which they are local minima.
\begin{lemma} \label{SecondVarSph}
Within the same setting as in Lemma \ref{FirstVarSph}, we have that
\begin{equation} \label{SecondVSph}
\begin{split}
\frac{d^2}{d \e^2}{\Big|}_{\e=0} \SSS (\phi_{\e}) = \int^T_0 \int_{\SS^n}  & [ 2
\drh^2r^n_0 + ((n+1)(n-2) + 2)\dr^2_0 r^{n-2}_0 \rh^2 + 2 (r^n_0)\dot{ }
(\rh^2)\dot{ } - \dr^2_0 r^{n-2}_0 |\hgrad \rh|^2 \\
                                                                  & + n^2
(n-2)(n-3) r^{n-4}_0 \rh^2 + (3n^2 - 8n) r^{n-4}_0 |\hgrad \rh|^2 + 2 r^{n-4}_0
(\hat{\D} \rh)^2] \rd \hmu \rd t \,.
\end{split}
\end{equation}
\end{lemma}
\begin{proof}
Adopting the same notation as in the proof of Lemma \ref{FirstVarSph}, we have
that $K_{\e} = \dr^2_{\eps} r^{n+1}_{\e} Q^{-1/2}_{\e}$. Consequently, we make the
following preliminary computations
\begin{equation*}
 \partialeo r^{n+1}_{\e} = (n+1) r^{n}_0 \rh, \qquad \partialeoo r^{n+1}_{\e} =
n (n+1) r^{n-1}_0 \rh^2 \,,
\end{equation*}

\begin{equation*}
 \partialeo \dr^2_{\e} = 2 \dr_0 \drh, \qquad \partialeoo \dr^2_{\e} = 2 \drh^2
\,,
\end{equation*}

\begin{equation*}
 \partialeo Q^{-1/2}_{\e} = - r^{-2}_0 \rh, \qquad \partialeoo Q^{-1/2}_{\e} =
r^{-3}_0 (2 \rh^{2} - |\hgrad \rh|^2) \,.
\end{equation*}
This way, we have that
\begin{equation} \label{SecondK}
	\begin{split}
		\partialeoo K_{\e} = & \, 2 r^n_0 \drh^2 + n (n+1) \dr^2_0 r^{n-2}_0 \rh^2 +
		\dr^2_0 r^{n-2}_0 (2 \rh^2 - |\hgrad \rh|^2) \\
                     & + 4 (n+1) \dr_0 r^{n-1}_0 \drh \rh - 4 \dr_0 r^{n-1}_0
		\drh \rh - 2 (n+1) \dr^2_0 r^{n-2}_0 \rh^2 \\
                   = & \, 2 r^n_0 \drh^2 + ((n-2)(n+1) + 2) \dr^2_0 r^{n-2}_0 \rh^2+
		4 n \dr_0 r^{n-1}_0 \drh \rh  - \dr^2_0 r^{n-2}_0 |\hgrad
		\rh|^2 \,.
	\end{split}
\end{equation}
For $P_{\e} = r^{n-5}_{\e} Q^{-5/2}_{\e} W^2_{\e}$, we compute
\begin{align*}
 	\partialeo r^{n-5}_{\e} &= (n-5) r^{n-6}_0 \rh,  \qquad \partialeoo r^{n-5}_{\e} =
		(n-5)(n-6) r^{n-7}_0 \rh^2 \,,\\
	 \partialeo Q^{-5/2}_{\e} &= - 5 r^{-6}_0 \rh, \qquad\partialeoo Q^{-5/2}_{\e} =
	5 r^{-7}_0 (6 \rh^2 - |\hgrad \rh|^2) \,,\\
	 \partialeo W_{\e} &= 4 n r_0^3\rho - r^3_0 \D \rh,   \qquad\partialeoo W_{\e} = 2
	r^2_0 [(n+1) |\hgrad \rh|^2 + 6 n \rh^2 - 3 \rh \hat{\D} \rh]\,,\\
	\partialeo W^2_{\e} &= 2 n r^7_0 (4 n \rh - \hat{\D} \rh), \\
	\partialeoo W^2_{\e} &=
	4 n r^6_0 [(n+1) |\hgrad \rh|^2 + 6 n \rh^2 - 3 \rh \hat{\D} \rh] + 2 r^6_0 (4 n
	\rh - \hat{\D} \rh)^2\,.
\end{align*}

By the previous computations,
we can conclude that
\begin{align} \label{SecondP}
	\partialeoo P_{\e} &= n^2 (n-2)(n-3) r^{n-4}_0 \rh^2 + (3n^2 - 8n) r^{n-4}_0
	|\hgrad \rh|^2 + 2 r^{n-4}_0 (\hat{\D} \rh)^2 +\\
	&\qquad + r_0^{n-4}(-4n^2+12n)\hgrad\cdot(\rh \hgrad\rh).
\end{align}
The thesis now follows integrating in space an time the sum of the equations
\eqref{SecondK} and \eqref{SecondP}, and using the boundary conditions imposed on $\rh$.
\end{proof}
For $n=2$ and the stationary spherical evolution $r_0$, by \eqref{EnergyConsSpheres} the integral over the third term in \eqref{SecondVSph} vanishes. Evaluating  \eqref{SecondVSph} in $r_0$ for spatially homogeneous $\rh$ we observe that the second variation is positive definite. This shows that for $n=2$ the spherical stationary point $r_0$ determined by \eqref{ODEStatSph} is the unique minimizer in the class of spherical evolutions. We therefore call $r_0$ in this case the $\SSS$-optimal spherical trajectory.

Equation \eqref{SecondVSph} drastically simplifies when $n=2$. In this case, we
can actually prove that the optimal spherical trajectory is not always a
minimizer of the action functional.
\begin{teo} \label{SecondVarSign}
Let $n=2$. Given two positive real numbers $R_0$ and $R_T$ the $\SSS-$optimal
spherical trajectory connecting the two concentric spheres of radii $R_0$ and
$R_T$ over the time interval $[0,T]$ is a local minimizer of $\SSS$ if $T\geq \frac{1}{3}\sqrt{3}T_{MCF}$, where $T_{MCF}$ was defined in \eqref{eq:TMCF}.

Furthermore, there exists $0<T_1\leq \frac{1}{3}\sqrt{3}T_{MCF}$ such that for $0<T< T_1$ the optimal spherical
trajectory connecting the given data is a not a local minimizer of $\SSS$. 
\end{teo}
\begin{proof}
When $n=2$, equation \eqref{SecondVSph} reads
\begin{equation} \label{SecondVSph2}
	\frac{d^2}{d \e^2}{\Big|}_{\e=0} \SSS (\phi_{\e}) = \int^T_0 \int_{\SS^2} [ 2
	\drh^2r^2_0 + 2 \dr^2_0 \rh^2- \dr^2_0 |\hgrad \rh|^2  + 2
	r^{-2}_0 ((\hat{\D} \rh + \rh)^2 - \rh^2)] d \hmu d t \,,
\end{equation}
where partial integration with respect to both spatial and time variables has
been performed and we have used that \eqref{Optimal2DS} implies
$(r^2_0)\,\ddot{ }=0$. We now choose $\rho(x,t) = \eta(t) \psi_l(x)$, where $\eta \in
C^{\infty}_0([0,T])$ is and $\psi_l : \SS^2 \rightarrow \R$ is the $l-$th
spherical harmonic associated to the standard metric on $\SS^2$. Substituting also $\dot{r}^2_0$ with its explicit expression given by equation \eqref{Optimal2DS} and recalling \eqref{eq:TMCF}, we get
\begin{equation} \label{SecondVSph3}
	\frac{d^2}{d \e^2}{\Big|}_{\e=0} \SSS (\phi_{\e}) = 4 \pi  \int^T_0 \Big[ 2
	{\dot{\eta}}^2 r^2_0 + \frac{\eta^2}{r_0^2}\Big(4(2-l(l+1))\frac{T_{MCF}^2}{T^2} + 2 ((l(l+1) - 1)^2 - 1)\Big) \Big] d t \,.
\end{equation}
One computes that the term in the large round brackets is for all $l\in \N_0$ nonnegative if $T\geq \frac{1}{3}\sqrt{3}T_{MCF}$. Since we can expand any perturbation in a series of spherical harmonics with time dependent coefficients and since the expression on the right-hand side of \eqref{SecondVSph2} splits in a sum of the corresponding expressions of the spherical harmonics, this shows that the second variation is positive definite for $T\geq \frac{1}{3}\sqrt{3}T_{MCF}$.\\
On the other hand, for any $\eta\in
C^{\infty}_0([0,T])$ and $l\geq 2$ fixed we can choose $0<T\ll 1$ such that the corresponding second variation becomes negative.
\end{proof}
\begin{Remark}
Theorem \ref{SecondVarSign} shows that for any pair of given concentric
spherical data, it is the amount of time we prescribe to join them which
determines whether the optimal spherical trajectory is a local minimum for the
action functional. What we have obtained can be actually heuristically
understood combining the "vanishing geodesic distance" result in \cite{MiMu06}
and the fact that for embedded surfaces in $\R^3$ the round spheres are the only
absolute minimizers for the Willmore functional. If we are given a long time
interval to connect the data, it will be convenient to pay a non optimal speed
contribution, being the curvature one always minimal. On short time intervals,
the possibility to make the speed contribution arbitrary close to zero will
compensate a non optimal curvature term.
\end{Remark}

We next complement the previous result and show a global minimizing property for the $\SSS$-optimal trajectory connecting two concentric spheres in $\R^3$.
\begin{teo} \label{ComparisonMCF}
Let $n=2$. Let $R_0 > R_T >0 $ and $T > 0 $ be positive real numbers. If  $T \geq
T_{MCF}(R_0 , R_T)$, the $\SSS-$optimal
spherical trajectory connecting the two concentric round $2-$spheres of radii
$R_0$ and $R_T$ over the time interval $[0,T]$ is a global minimum for $\SSS$ in the class of smooth evolutions.
\end{teo}
\begin{proof}
We first observe that for any $c \in \R$ we have
\begin{equation*}
 8 \pi (R_0^2 - R_T^2) = 2 \int_0^T\int_M v \HHH d \bmu_t d t =
\int_0^T\int_M \Big( - \frac{1}{c} (v - c \HHH)^2 + \frac{1}{c} v^2 + c \HHH^2 \Big) d
\bmu_t d t.
\end{equation*}
Thus, we obtain
\begin{equation} \label{GlobalOptimalSpheres}
\begin{split}
\int_0^T\int_M (v^2 + \HHH^2) \rd \bmu_t \rd t & = 8 \pi c (R_0^2 - R_T^2) +
\int_0^T\int_M (v - c \HHH)^2 d \bmu_t d t + (1-c^2) \int_0^T \HHH^2
d \bmu_t d t \\
                                                  & \geq 8 \pi c (R_0^2 - R_T^2)
+ (1-c^2) \int_0^T \int_M \HHH^2 d \bmu_t d t \,,
\end{split}
\end{equation}
where the equality holds if and only if
\begin{equation} \label{OptimalCase}
v= c \HHH
\end{equation}
at each point in space and time. Moreover, if \eqref{OptimalCase} holds, the
initial and final conditions force the solution to be a spherical one. Since
$\int_M \HHH^2 d \bmu_t\geq 16\pi$, as spheres are the unique minimizer of the Willmore energy for smooth embeddings of $M$, we deduce from \eqref{GlobalOptimalSpheres} that 
\begin{equation*}
	\int_0^T\int_M (v^2 + \HHH^2) \rd \bmu_t \rd t \,\geq\, \max_{c^2\leq 1}  \Big(8 \pi c (R_0^2 -
			R_T^2) + 16 (1-c^2) \pi T\Big).
\end{equation*}
Explicit calculations show that the maximum on the right hand side is uniquely attained for $c_* = \frac{T_{MCF}}{T}$ (note that by assumption $c_*\leq 1$), and that we thereby obtain the estimate
\begin{equation}
	\int_0^T\int_M (v^2 + \HHH^2) d \bmu_t d t \,\geq\, 16\pi\Big( \frac{T_{MCF}^2}{T} +T\Big). \label{eq:opt-S-sm}
\end{equation}
The value of the right-hand side coincides with the value of the action functional of $r_0$, which proves the optimality.
\end{proof}
The proof of the preceding theorem shows that the $\SSS$-optimal spherical trajectory is optimal also in the class of non-vanishing evolutions that are piecewise smooth and where $t\mapsto \bmu_t(M)$ is continuous. In a final remark we compare the $\SSS$-optimal spherical trajectory with the non-smooth evolution that vanishes for some positive time interval.
\begin{Remark}
Let the sphere $R_0$ evolve by mean curvature flow, until it vanishes at $T_{MCF}(R_0)=\frac{R^2_0}{4}$. Consider then a point nucleation at $T - T_{MCF}(R_T) = T - \frac{R^2_T}{4}$, which evolve by time--reversed MCF up to time $T$, to get the sphere of radius $R_T$ at time $T$ (notice that this kind of trajectory makes sense under the further assumption $T \geq \max\{T_{MCF}(R_0),T_{MCF}(R_T)\}$). The corresponding value for the action is given by $8 \pi (R^2_0 + R_T^2)$, which is twice the sum of the area of the initial and final datum. This has to be compared with the value of the $\SSS$-optimal smooth spherical solution, for which after \eqref{eq:opt-S-sm} we have found the value $16\pi\Big( \frac{T_{MCF}^2}{T} +T\Big)$. Comparing both expressions shows that for $T>\frac{(R_0+R_T)^2}{4}$ the non-smooth trajectory has lower action.

Collecting all the results, the following scenario arises. For $T < T_{MCF}(R_0,R_T)$ a solution to the minimal action problem exists, but we can not say if it is smooth and if it is spherically symmetric. For $T_{MCF}(R_0,R_T) \leq T \leq \frac{(R_0+R_T)^2}{4}$, the optimal smooth spherically symmetric connection is the absolute action minimizer. For $T > \frac{(R_0+R_T)^2}{4}$ the optimal smooth rotationally symmetric connection is still a local minimizer, nevertheless the absolute minimum of the action functional is attained at the connection with point nucleation, which is anyway rotationally symmetric, in accordance with Theorem \ref{ComparisonMCF}. Notice also that the energy of the connection with point nucleation is constant (in accordance to energy conservation) and equal to zero.
\end{Remark}
\begin{appendix}
\iflong
\section{Notations and results from geometric measure theory}
In this section we will fix the notation and recall some results from geometric
measure theory which will be used along the paper.

\subsection{Integer-rectifiable Radon measures with Mean Curvature}

\begin{Definition}
 Let $\mu$ be a Radon measure on $\R^{n+1}$. We say that $\mu$ has an
$n-$dimensional tangent plane at a point $x \in \R^{n+1}$ if there exist a
number $\Theta > 0$ and an $n-$dimensional hyperplane $T \in \R^{n+1}$ such that
\begin{equation}
 \lim_{r \rightarrow 0} \frac{1}{r^n}\int_{\R^{n+1}} \eta \Big( \frac{y-x}{r}
\Big) d \mu (y) = \Theta \int_T \eta \, d \H^{n} , \qquad \forall \eta \in
C^0_c(\R^{n+1}) \,.
\end{equation}
In this case, we set $T_x \mu = T$ and we call $\Theta$ the multiplicity of
$\mu$ at $x$.

If $\mu$ has an $n-$dimensional tangent plane $\mu-$almost everywhere on
$\R^{n+1}$, we will say that $\mu$ is $n-$rectifiable (or simply rectifiable).
If $\Theta$ is $\mu-$almost everywhere integer, we say that $\mu$ is
integer-rectifiable.

For a rectifiable Radon measure $\mu$ on $\R^{n+1}$ the first variation $\d \mu:
C^1_c(\R^{n+1} ; \R^{n+1}) \rightarrow \R$ of $\mu$ is given by

\begin{equation}
 \d \mu (\zeta) = \int_{\R^{n+1}} \div_{T_x \mu} \zeta(x) d \mu(x) \qquad
\forall
\zeta \in C^1_c(\R^{n+1} ; \R^{n+1})\ ,.
\end{equation}

Moreover, if there exists a function $\HHH \in L^1_{loc}(\mu ; \R^{n+1})$ such
that

\begin{equation}
\d \mu (\zeta) = - \int_{\R^{n+1}} \HHH \cdot \zeta d \mu \qquad \forall
\zeta \in C^1_c(\R^{n+1} , \R^{n+1}) \,,
\end{equation}

we will say that $\mu$ has weak mean curvature. 

\end{Definition}

\subsection{Measure-function pairs} In order to establish a regularity result
for minimizers of $\SSS$, we will need the notion of \emph{measure-function pair
convergence} introduced in (\cite{Hutc86}). 
\begin{Definition}
Given a Radon measure $\mu$ on $\R^n$ 
and a function $f\in L^1_{\loc}(\mu,\R^l)$, $l\in\N$, we say that $(\mu , f)$ is a
\emph{measure-function pair}. \\
A sequence $(\mu^k, f_k)_{k\in\N}$ of measure function pairs converge weakly to the measure-function $(\mu,f)$ if
\begin{equation}
 	\int_{\R^n} \op f_k , \xi \cl d \mu^k \rightarrow \int_{\R^n} \op f , \xi
	\cl d \mu
\end{equation}
for all $\xi\in C^0_c(\R^n,\R^l)$.

\end{Definition}
The following result, which is proved in \cite{Hutc86}, gives a criterion
for weak converge of sequences of \emph{measure-function pairs} as well as
necessary and sufficient condition for the convergence to be strong.
\begin{prop} \label{L2Hutchinson}
If $(\mu^k, f_k)_{k\in\N}$ is a sequence of \emph{measure-function pairs} satisfying
\begin{equation*}
 \sup_{k \in \N} ||f_k||_{L^2(\mu_k)} < + \infty
\end{equation*}
and if $\mu$ is a Radon measure on $\R^n$ such that $\mu^k \rightarrow \mu$,
then
there exist a function $f \in L^2(\mu)$ and a subsequence $k \rightarrow
+\infty$ such that $(\mu^k,f_k)_{k\in\N}$ converges weakly to $(\mu,f)$ as
\emph{measure-function pairs}.

Moreover, if the sequence
$\{(\mu^k,f_k)\}$ weakly converges to $(\mu,f)$ the lower semicontinuity property
\begin{equation}
	 ||f||_{L^2(\mu)} \leq \liminf_{k \rightarrow + \infty} ||f_k||_{L^2(\mu_k)}
\end{equation}
holds.
\end{prop}

\fi
\section{Notations and results from differential geometry}\label{AppDG}

Let $M$ be an $n$-dimensional smooth
differentiable manifold without boundary and $\phi : M \rightarrow
{\mathbb{R}}^{n+1}$ a smooth immersion. Denoting with $(x^1 , ... , x^n)$ a
local coordinate system on $M$ and $(\frac{\partial}{\partial x_1}
, ... , \frac{\partial}{\partial x_n}) : = (e_1 , ... , e_n)$ the associated
base for the tangent space, the Riemannian metric $g$
naturally induced by $\phi$ on $M$ via the pullback reads as follows:

\begin{equation} \label{InducedMetric}
g_{ij} := \op \partial_i \phi , \partial_j \phi \cl \,,
\end{equation}

where $\op \cdot , \cdot \cl$ denotes the standard scalar product in
${\mathbb{R}}^{n+1}$.

We will denote with $\grad$ the covariant derivative associated to the
Levi-Civita connection of $g$.

\noindent Since $\phi(M)$ has codimension one in ${\mathbb{R}}^{n+1}$, it
follows that $M$ is orientable and at each point of $\phi(M)$ there is a well
defined (up to sign) normal vector  field that we call $\nu$. Within this
setting we define (giving components) the second fundamental form of $\phi(M)$
according to

\be \label{SecondFF}
\AAA = \hhh_{ij} := \op \nu , \partial^2_{ij} \phi\cl \,,
\ee

\noindent which immediately implies that $\AAA$ is a well defined symmetric
2-tensor on $\phi(M)$.

\noindent The mean curvature of the couple $(M,\phi)$ is defined as the trace of
the second fundamental form and will be denoted by $\HHH$. We will often denote
the mean curvature vector $\HHH \nu$ just with $\HHH$.

\noindent Within this setting, the Gauss-Weingarten relations read

\be \label{GaussWein}
\partial^2_{ij} \phi = \G^k_{ij} \partial_k \phi + \hhh_{ij} \nu \quad
\textrm{and} \quad \partial_i \nu = - \hhh_{ik} g^{kl} \partial_l \phi \,.
\ee

\noindent The Bianchi identities for the curvature tensor of the
immersed manifold are equivalent to

\begin{equation} \label{ImmersedBianchi}
\grad_i \hhh_{jk} = \grad_j \hhh_{ik} \,.
\end{equation}

\end{appendix}


\end{document}